\documentclass[11pt,a4paper]{article}
\usepackage{bbm}

\usepackage{graphicx, color, tikz}
\graphicspath{ {./images/} }
\usepackage{wrapfig}
\usepackage{multicol}
\usepackage[mathscr]{euscript}
\usepackage{subcaption}
\usepackage{hyperref}
\usepackage{epstopdf}
\usepackage{enumerate}
\usepackage{amsmath,amsfonts,amssymb,amsthm,epsfig,epstopdf,titling,url,array}
\usepackage{color}
\usepackage{framed}
\usepackage[utf8]{inputenc}
\usepackage[english]{babel}

\usepackage{xparse}
\usepackage[margin=2.5cm]{geometry}
\usepackage[most]{tcolorbox}
\usepackage[export]{adjustbox}

\NewDocumentCommand{\INTERVALINNARDS}{ m m }{
	#1 {,} #2
}
\NewDocumentCommand{\interval}{ s m >{\SplitArgument{1}{,}}m m o }
{
	\IfBooleanTF{#1}{
		\left#2 \INTERVALINNARDS #3 \right#4
	}{
		\IfValueTF{#5}{
			#5{#2} \INTERVALINNARDS #3 #5{#4}
		}{
			#2 \INTERVALINNARDS #3 #4
		}
	}
}

\newtheorem{theorem}{Theorem}[section]
\newtheorem{corollary}{Corollary}[theorem]
\newtheorem{definition}{Definition}[section]
\newtheorem{lemma}[theorem]{Lemma}
\newtheorem{example}{Example}[section]
\newtheorem{prop}[theorem]{Proposition}

\newtheorem{Remark}{Remark}[section]

\newtheorem{Question}{Question}

\usepackage{authblk}
\begin{document} 
	\title{\textbf{A Survey of Baker Wandering Domains}}
	\author{Sukanta Das\footnote{Corresponding author, a21ma09005@iitbbs.ac.in}}
	\author[2]{Tarakanta Nayak\footnote{tnayak@iitbbs.ac.in}}
	\affil[1,2]{Department of Mathematics, Indian Institute of Technology Bhubaneswar, Bhubaneswar, India}
	\date{}
	\maketitle
	\begin{abstract}
		Let $f:\mathbb C\to \widehat{\mathbb C}=\mathbb C \cup\{\infty\}$ be a transcendental meromorphic function (possibly without any pole) with a single essential singularity, and that is chosen to be at $\infty$. The set of points $z\in\mathbb{{\widehat{C}}}$ such that the family of iterates $\{f^n\}_{n\geq 0}$ is defined and forms a normal family in a neighborhood of $z$ is known as the Fatou set of $f$. For a Fatou component $W$, let   $W_j$ denote the Fatou component containing $f^j(W)$. A Fatou component $W$ is called wandering if $W_m\bigcap W_n=\emptyset$ for all $m \neq n$.  A wandering domain $W$ of  $f$ is called a Baker wandering domain, if each $W_n$ is bounded, multiply connected, and $W_n$ surrounds $0$ for all large $n$ and, dist$(W_n,0)\to\infty$ as $n\to\infty$.
		This paper surveys the current state of knowledge on Baker wandering domains. We revisit the first example of the Baker wandering domain, followed by other examples. The influence of Baker wandering domain on the singular values and dynamics of the function is presented.  We also discuss some classes of functions that do not possess any Baker wandering domain. Several problems are proposed throughout the article at relevant places.
	\end{abstract}
	\textit{Keywords:}
	Transcendental meromorphic functions; Baker wandering domains; Singular values; Baker omitted value.\\
	
	AMS Subject Classification: 37F10, 30D05
	
	\section{Introduction}
	Let $f:\mathbb C\to \widehat{\mathbb C}=\mathbb C \cup\{\infty\}$ be a transcendental meromorphic function with a single essential singularity, and that is chosen to be at $\infty$. The only other type of possible singularity of such a function is a pole. Throughout this article, by a meromorphic function, we mean a transcendental meromorphic function. Such a function can be without any poles, and in that case, it is called entire. 
	Depending on the behaviour of the iterates $f^n$ ($n-$times composition of $f$ with itself), the Riemann sphere $\mathbb{{\widehat{C}}}$ is partitioned into two sets: the Fatou set and the Julia set. The Fatou set of $f$ is the set of all points in $\mathbb{{\widehat{C}}}$ such that the family $\{f^n\}_{n\geq 1}$ is normal in a neighborhood of $z$. The Julia set, denoted by $\mathcal{J}(f)$, is the complement of the Fatou set in $\widehat{\mathbb{C}}$. If $f^n$ is not defined at a point $z$ for some $n 
	\geq 1$, then we take $z$ to be in the Julia set. More precisely, the point $\infty$ and its backward orbit $\{z: f^n (z) =\infty ~\mbox{for some}~n \geq 0 \}$ is contained in the Julia set of every meromorphic function. In particular, all the poles, whenever these exist are in the Julia set. We say a subset $A$ of $\widehat{\mathbb{C}}$ is completely invariant under $f$ if it is forward-invariant, i.e., $f(A) \subseteq A$, as well as backwards-invariant, i.e., $f^{-1} (A) \subseteq A$. The Fatou set as well as the Julia set are completely invariant under the function. A maximally connected subset of the Fatou set is called a Fatou component. For a Fatou component $V$, we denote the Fatou component containing $f^k(V)$ by $V_k$ for each $k\geq 0$, where $V_0=V$. A Fatou component $V$ is periodic if $V_p =V$ for some $p$, where the smallest such $p$ is known as the period of $V$. A periodic Fatou component can be an attracting domain, a parabolic domain, a Siegel disk, a Herman ring or a Baker domain.  Further details on the periodic Fatou components can be found in \cite{Berg 1993}.
	\par  A $p$-periodic Fatou component $V$ is called a Baker domain of $f$ if $\lim_{n \to \infty}f^{np} (z) = z^*$ for all $z \in V$ where $z^*$ is an essential singularity of $f^p$. Here $z^*$ is such that $f^k (z^*)=\infty$ for some $0 \leq k \leq p-1$. It is indeed a fact that $V_k$ is unbounded for some $k$ whenever $V$ is a Baker domain (see Theorem 13, \cite{Berg 1993}). Rational maps (without any essential singularity) cannot have any Baker domain in their Fatou sets.

	A Fatou component $W$ is a wandering domain if $W_m\bigcap W_n=\emptyset$ for all $m\neq n$.  The connectivity of a Fatou component $U$, denoted by $c(U)$, is the number of components of $\widehat{\mathbb{C}} \setminus U$. We say $U$ is multiply connected whenever $c(U)>1$. In 1963, Baker constructed an entire function $f$ possessing a multiply connected Fatou component \cite{Baker 1963}. He constructed a sequence of concentric annuli  $\{A_n\}_{n\geq 1}$ such that $A_{n+1}$ surrounds $A_n$, $f(A_n)\subseteq A_{n+1}$ for all $n$, and   $\lim_{n \to \infty} f^n(z) = \infty$ for all $z\in A_1$. Here and now onwards, for two subsets $A,B$ of $\mathbb{C}$, we say $A$ surrounds  $B$ if there exists a bounded component of $\mathbb{{\widehat{C}}}\setminus A$   containing $B$. Each $A_n$ was shown to be contained in a multiply connected Fatou component of $f$. At that time, Baker couldn't assert whether the Fatou components containing these distinct annuli are distinct or not. Later in 1976, he himself proved that these are distinct \cite{Baker 1976}. This is the discovery of wandering domains. Before a year in 1975, Baker established that a transcendental entire function cannot have any unbounded multiply connected Fatou component \cite{Baker 1975}. In particular, the wandering domains constructed by him were bounded. 
	Later in 1985, Sullivan  proved that rational maps cannot have any wandering domain (see Theorem 1, \cite{Sullivan 1985}). 
	After that, the search for wandering domains completely shifted to transcendental functions.
	\par  A wandering domain of an entire function can be simply or multiply connected (see~\cite{Baker 1976,Baker 1984}), and bounded or unbounded (see \cite{EremenkoLyubich 1987}). The limit functions of $\{f^n\}_{n>0}$ on a wandering domain are always constant (see Lemma 2.1, \cite{bk3}). Let $L(W)$  denote the set of all the limit functions of $\{f^n\}_{n>0}$ on a wandering domain $W$. A wandering domain is called escaping, oscillating or dynamically bounded if $L(W) =\{\infty\}, $ contains $\infty$ and at least one finite complex number or is a bounded set, respectively. The first example of wandering domains, given by Baker as mentioned above, are  multiply connected, bounded  and escaping. This is the motivation for the following definition.  
	
	\begin{definition}[Baker wandering domain]
		A wandering domain $W$ of a meromorphic function $f$ is called a Baker wandering domain, \textit{BWD} in short,  if for all sufficiently large $n$, $W_n$ is multiply connected, bounded and surrounds $0$ such that $f^n(z)\to \infty$ as $n\to \infty$ for all $z\in W$.
	\end{definition}
	Though the first example of BWD was for an entire function, these can actually coexist along with poles (see Subsection~\ref{meromorphic}). In 2000, Rippon and Stallard introduced the name {\it{Baker wandering domain}} (see \cite{RipponStallard 2000}) after 37 years of its construction. This is the primary object of this article. 
	
	\par A function with a BWD, from a dynamical perspective, is far from being simple. To make this more precise, we recall some definitions. A complex number $c$ is called a critical point of $f$ if  $f'(c)=0$ or $c$ is a multiple pole of $f$. A critical value of $f$ is the image of a critical point. For some $a \in \widehat{\mathbb{C}}$, if  $\lim_{t \to \infty}f(\gamma(t)) =a$  along a curve $\gamma:[0,\infty)\to\mathbb{C}$ such that $\lim_{t \to \infty}\gamma(t) = \infty$, then $a$ is called an asymptotic value of $f$. A singular value is a critical value, an asymptotic value or a limit point of these values. The set of all singular values of $f$ is denoted by $sing(f^{-1})$ and is well-known to control several aspects of the dynamics of $f$. A detailed treatment of singular values is presented in Subsection~\ref{singularvalues-section}.
	\par  The famous Eremenko-Lyubich class, denoted by $\mathcal{B}$ is the set of all entire functions for which $sing(f^{-1})$ is bounded. Most of the research undertaken so far on dynamics of meromorphic functions is focused on $\mathcal{B}$. One main tool developed and extensively used for studying these functions is the logarithmic change of variables (see \cite{EremenkoLyubich 1992,Rempe2017, RRRS 2011, Sixsmith2018}). This idea has been developed further by Rippon and Stallard to study the dynamics of meromorphic functions (see \cite{RipponStallard 1999}). However, this tool does not work when $sing(f^{-1})$ is unbounded. To the best of our knowledge, functions of unbounded type are not well-understood in any reasonable way. It seems desirable to start with a \textit{tractable} subclass - to be made precise soon, and this is where BWDs come into picture.
	\par  If $f$ is a meromorphic function with BWD then $sing(f^{-1})$ is unbounded. A proof for entire functions can be found in Lemma 2.5, \cite{GhoraNayakSahoo 2021}. Theorem A, \cite{RipponStallard 1999} (for $n=1$)  gives that if  the set of all the finite singular values of $f$ is bounded then there is no Fatou component $U$ such that $\lim_{m \to \infty}f^m (z) =\infty$ for $z \in U$. But on every BWD, $f^m \to \infty$ as $m \to \infty$ (see also Theorem \ref{B-NoBWD}). In spite of this, there is an intrinsic boundedness, and the effect of the essential singularity is quite limited. We now elaborate this. 
	\begin{itemize}
		\item \textit{Bounded Fatou components:}
		A BWD is not only bounded, but its very presence ensures the boundedness of all other Fatou components, whenever they exist.  An immediate consequence is the absence of Baker domains since each Baker domain itself or some of its iterated forward images is unbounded, as already observed in the second paragraph of this section. This is one of the important similarities with rational maps. Further, the function restricted to each of its Fatou components (including BWD) is \textit{proper} and that allows the use of the tools and techniques used for understanding the dynamics of rational maps. An important example of this phenomena is the study of connectivity of Fatou components using the \textit{Riemann-Hurwitz formula}.
		\item 	 \textit{Bounded Julia components:} 	A maximally connected subset of the Julia set is called a Julia component. In the presence of a BWD, the essential singularity (i.e., $\infty$) becomes a singleton component of the Julia set and is also a buried point, i.e., it is not in the boundary of any  Fatou component. In fact, every point in the backward orbit of $\infty$, i.e., $\{z \in \widehat{\mathbb{C}}: f^n (z)=\infty~\mbox{for some}~n\geq 0\}$ is a singleton and buried component of the Julia set whenever there is a BWD (see Proposition~\ref{pre-poles}). Further, this set is a completely invariant proper subset of the Julia set. In this way, the effect of the essential singularity on the dynamics of the function with a BWD becomes considerably limited.
		\item   \textit{No finite asymptotic value:} There is no finite asymptotic value (see Proposition~\ref{unbounded-BWD}). Therefore $sing(f^{-1})$ is actually the set of all critical values of $f$ and their limit points.
	\end{itemize}  
	To summarize, functions with BWD can be the right candidates to start with for a systematic investigation of dynamics of unbounded type functions. This necessitates putting  all the important known facts on BWD together and in proper context. This is the motivation for this article. 
	\par We start with a detailed discussion of the first example of BWD given by Baker, followed by other examples demonstrating similarities and differences with the former. Then the influence of BWD on the singular values and dynamics of the concerned function is presented. More precisely, we discuss in detail the singularity lying over $\infty$ and take up Julia components, the escaping set (A point is called escaping for a meromorphic function $f$ if $\lim_{n \to \infty}f^n(z) = \infty$ but $f^n(z)\neq\infty $ for any $n$) and eventual connectivity of BWDs for discussion. After that, functions that do not have any BWD are discussed. An important point to note is that most of the results known so far in this context are on entire functions or on meromorphic functions with finitely many poles.  A set of well-framed questions are also presented at suitable places that may guide the future direction of research.
	
	\par The structure of the paper is as follows. In Section 2, we provide examples of meromorphic functions with Baker wandering domains along with some sufficient conditions for BWD. Section 3 discusses  several implications of Baker wandering domains in terms of the singular values and dynamics of the concerned function. Finally, Section 4 presents several criteria for functions ensuring non-existence of BWD. 
	\par Throughout this article, $A(r,R)$ represents the annulus $A(r, R)=\{z \in \mathbb{C}:r<|z|<R\}$. 
	\section{Examples and Sufficient conditions }
	Since the presence of poles gives rise to situations that are qualitatively different from those in the absence of any pole, we present the examples of Baker wandering domains in these two situations separately. All the examples are based on the first one constructed by Baker, which we revisit at the beginning of this section. We indicate and correct some minor errors in calculations in Baker’s example \cite{Baker 1963, Baker 1984}. However, this does not affect the validity of the main arguments.
	
	\subsection{Entire functions}
	The following result published in 1984 provides a complete description of a BWD.
	\begin{theorem}(Theorem 3.1,\cite{Baker 1984}) \label{Thm1}
		If $U$ is a multiply connected Fatou component of an entire function $f$ then, $U$ is a wandering domain and $f^n(z)\to\infty$ for all $z\in U$ as $n\to\infty$. Further, for all sufficiently large $n$, $U_n$ contains a closed curve $\gamma_n$ whose distance from $0$ tends to $\infty$ as $n\to\infty$ and whose winding number from $0$ is non-zero. In this case, every Fatou component of $f$ (including those different from $U_n$) is bounded.	
	\end{theorem}
	The proof of Theorem \ref{Thm1} uses the Maximum Modulus and the Argument Principle. In this theorem, each $U_n$ is bounded and multiply connected such that $U_{n+1}$ surrounds $U_n$ as well as $0$ for all large $n$, and $f^n(z)\to\infty$ as $n\to\infty$ for every $z\in U$. Hence, $U$ is a BWD. The converse is also true. 
	\begin{corollary}\label{BWD-iff-MCFC}
		A Fatou component of an entire function is a BWD if and only if it is multiply connected. In particular, a single bounded Julia component implies the existence of a BWD.
	\end{corollary}
	There is an easy lemma that is to be used in the examples.
	\begin{lemma}
		Let $f$ be entire and $r_n$ be a sequence of positive real numbers such that $\lim_{n \to \infty}r_n =\infty$. If   $A_n = A(r_n, R_n)$ for $ R_n >r_n $ denotes the annulus and $f(A_n) \subseteq A_{n+1}$ for each $n \geq 1$, then there is a BWD containing $A_n$.
		\label{BWD-lem}	\end{lemma}
	\begin{proof}
		By the Fundamental Normality Test (Theorem 3.3.4, \cite{Beardon-Book}), each annulus $A_n$ is contained in the Fatou set of $f$ since $\{f^n\}_{n>0}$ omits all the points surrounded by $A_1$. The Julia set of $f$ is non-empty since there are repelling periodic points and those are always in the Julia set (see Theorem 1, \cite{Berg 1993}). By Picard's theorem and the complete invariance of the Julia set, each neighborhood of $\infty$ contains a point of the Julia set. This gives that the Fatou component   $U_n$  containing $A_n$ is multiply connected for a sufficiently large $n$. Therefore, $U_n$ is a BWD by Corollary~\ref{BWD-iff-MCFC}.	\end{proof}
	
	We now present the first example of a BWD that is constructed by Baker \cite{Baker 1963}. 	
	\begin{example}\label{1st exm}
		Let $C=\frac{1}{4e}$ and $r_1> 4e$. Also, let $r_2=2Cr^2_1$ and for each $n>2$,  \begin{equation}\label{Eq1}
			r_{n+1}=Cr_n^2\prod_{i=1}^{n}\left(1+\frac{r_n}{r_i}\right).
		\end{equation}	
		Then the function \begin{equation}\label{Baker-first}
			f(z)=Cz^2\prod_{i=1}^{\infty}\left(1+\frac{z}{r_i}\right) 
		\end{equation} has BWD.		
	\end{example}
	\begin{proof}

		From Equation (\ref{Eq1}), it is easy to see that \begin{equation}\label{Eq2}
			\frac{r_{n+1}}{r_n}>Cr_n ~\mbox{for all }~ n.
		\end{equation} 
			Indeed, $r_n$ increases \textit{very fast}, to be made precise soon, and this is the key to the desired property of the function.
			\par Observe from Equation (\ref{Baker-first}) that for each $z$, we have
			\begin{equation}\label{Eq5}
				|f(z)|\leq f(|z|) ~\mbox{and}~ |f(z)|\geq  f(-|z|).
			\end{equation}
			It follows from Equation (\ref{Eq1}) that $r_{n+1}\geq 2Cr^2_n>2Cr_1r_n>2r_n$ for all $n\geq 1$. In general, 
			\begin{equation}
				r_{n+j}>2^{j}r_n ~\mbox{for all}~ j, n\geq 1.
				\label{growth-rn}\end{equation}
			We get $\frac{1}{r_i}<\frac{1}{r_12^{i-1}}$ and consequently, $\sum_{i=1}^{\infty}\frac{1}{r_i}<\frac{1}{r_1}\sum_{i=1}^{\infty}\frac{1}{2^{i-1}}=\frac{2}{r_1}$. The infinite product in (\ref{Baker-first}) converges if and only if $\sum_{i=1}^{\infty}\left|\frac{z}{r_i}\right|$ converges on every compact subsets of $\mathbb{C}$ (see Chapter VII in \cite{Noop 1947}). This is actually the case since $\sum_{i=1}^{\infty}\frac{1}{r_i}$ converges. Therefore, the function $f$ is entire. Some useful properties of this function are now listed.
			\par Property 1. For $|z|\leq 1$, $|f(z)|<\frac{1}{4}|z|^2$ and in particular, $|f(z)|<\frac{1}{4}|z|$;
			\par Property 2. $r_{n+1}<f(r_n)<er_{n+1}$;
			\par Property 3. $f({\sqrt{r_n}})<{\sqrt{r_{n+1}}}$;
			\par Property 4. $f({r^2_n})>4{r^2_{n+1}}$.
			
			\par The first is the contracting property of $f$ near the origin, whereas the last three describe the growth of $f$ along the sequence $\{r_n\}_{n\geq 1}$ and other associated sequences. 
			\par We first show Property 1. For $|z| \leq 1$, $|f(z)| \leq C |z|^2 \prod_{i=1}^{\infty}\left(1+\frac{1}{r_i}\right)<C |z|^2 \prod_{i=1}^{\infty}\left(1+\frac{2^{1-i}}{r_1}\right)$. Now $\log\prod_{i=1}^{\infty}(1+2^{1-i}r_1^{-1})=\sum_{i=1}^{\infty}\log(1+2^{1-i}r_1^{-1})<\sum_{i=1}^{\infty}2^{1-i}r_1^{-1}=\frac{2}{r_1}$. Here, we use $\log(1+x)<x$ for $x>0$. Therefore, for all $|z|\leq 1$, $|f(z)|< Ce^{\frac{2}{r_1}}|z|^2 <  \frac{1}{4e}e^{\frac{1}{2e}}|z|^2=\frac{1}{4} e^{\frac{1}{2e}-1}|z|^2< \frac{1}{4}|z|^2$. For showing Property 2, observe that $r_{n+1}=Cr_n^2\prod_{i=1}^{n}\left(1+\frac{r_n}{r_i}\right)<f(r_n)=r_{n+1}\prod_{i=n+1}^{\infty}\left(1+\frac{r_n}{r_i}\right)<	r_{n+1}\prod_{i=n+1}^{\infty}\left(1+\frac{1}{2^{i-n}}\right)$, the last inequality following from Inequality~(\ref{growth-rn}). Since $\prod_{i=n+1}^{\infty}\left(1+\frac{1}{2^{i-n}}\right) = \prod_{i=1}^{\infty}\left(1+\frac{1}{2^i}\right)<e 
			$ (can be verified by taking the logarithm),
			we have 	
			\begin{equation}
				r_{n+1} < f(r_n) <er_{n+1}.
				\label{f(rn)-bound}
			\end{equation}	
			This is Property 2.
			Baker mistakenly took  $\prod_{i=1}^{\infty}\left(1+\frac{1}{2^i}\right)=2$, which however does not affect the final conclusion. 
			\par Properties 3 and 4 arise from Hadamard's Three Circle Theorem (see \cite{Hadamard1896}). It follows from the first inequality of Inequality~(\ref{Eq5}) that the maximum modulus $M(r,f)$ of $f$ on $\{z: |z|\leq r\}$ is attained at $z=r$ i.e., $M(r,f)=f(r)$. Further,   $f $ is an increasing function of $r$ by the Maximum Modulus Principle. By Hadamard's Three Circle Theorem, we have for $0<r_1<r<r_2$, \begin{equation}\label{HTC}
				\log f(r)<\alpha \log f(r_1)+(1-\alpha)\log f(r_2) ~\mbox{where}~ \alpha=\frac{\log\left(\frac{r_2}{r}\right)}{\log\left(\frac{r_2}{r_1}\right)}.
			\end{equation}
			For $s>0$, putting $r=e^s$, $r_1=1$ and $r_2=e^{2s}$ in Inequality (\ref{HTC}), we have $\alpha=\frac{1}{2}$ and,  $2\log f(e^s)<\log f(1)+\log f(e^{2s}).$  In other words, $f(r^2)>\frac{f(r)^2}{f(1)}$. Since $|f(z)|<\frac{1}{4}$ for all $|z|\leq 1$ (by Property 1), we have $\frac{1}{f(1)}>4$ and consequently 
			\begin{equation}\label{f(r)-inequality}
				f(r^2)>4f(r)^2~\mbox{for all} ~r>1.
			\end{equation} Putting $r=\sqrt{r_n}$, we get $ f(\sqrt{r_n})< \frac{1}{2}\sqrt{ f(r_n)}$. Using the second inequality of Property 2, we have $$ \frac{1}{2}{\sqrt{f(r_n)}}<\frac{\sqrt{e}}{2}\sqrt{r_{n+1}}<\sqrt{r_{n+1}}.$$ Therefore, $f({\sqrt{r_n}})<{\sqrt{r_{n+1}}}$, which is nothing but Property 3.
			Similarly, putting $r=r_n$ in Inequality (\ref{f(r)-inequality}) and using the first inequality of Property 2, we get 
			$$f({r^2_n})>4{f({r_n})}^2>4{r^2_{n+1}}.$$ This is Property 4.
			
			\par Finally, let $A_n=\{z:r_n^2<|z|<\sqrt{r_{n+1}}\}$ for each $n \geq 1$. We first show that each $A_n$ is non-empty for all large $n$. For this, it will suffice to establish $ \lim_{n \to \infty}\frac{\sqrt{r_{n+1}}}{r_n^2} =\infty$, which is  equivalent to $\lim_{n \to \infty}\frac{r_{n+1}}{r_n^4} = \infty$. Observe that 
			\begin{equation}
				\frac{r_{n+1}}{r_n^4}=\frac{2C}{r^2_n}\prod_{i=1}^{n-1}\left(1+\frac{r_n}{r_i}\right) >\frac{2C}{r^2_n}\frac{r^{n-1}_n}{r_1r_2\ldots r_{n-1}} >\frac{2Cr^{n-3}_n}{  (r_{n-1})^{n-1}}.\\
				\label{quotient-rn}\end{equation}
			By Inequality (\ref{Eq2}), the last term is bigger than 
			$ 2C^3 (C r_{n-1})^{n-5},$  which clearly goes to $\infty$ as $n\to\infty$.	
			\par We now establish  $f(A_n) \subset A_{n+1}$, i.e., $r_{n+1}^2 < |f(z)| <\sqrt{r_{n+2}}$ for all large $n$ and for all $z \in A_n$.
			It follows from the first part of Inequality (\ref{Eq5}) and Property 3 that, for $z\in A_n$, $$|f(z)|\leq f(|z|)<f({\sqrt{r_{n+1}}})<{\sqrt{r _{n+2}}} ~\mbox{for all sufficiently large}~ n.$$
			We shall be done by showing that for all $z \in A_n$,  $|f(z)| > r_{n+1}^2$ for all large $n$. For each $z \in A_n$, observe that $4r_n < |z| < \frac{r_{n+1}}{4}$.
			

			For $|z|=r$, consider $\log \left|\frac{f(r)}{f(-r)}\right|=\sum_{k=1}^{\infty} I_k ~\mbox{where}~ I_k= \log\left|\frac{1+\frac{r}{r_k}}{1-\frac{r}{r_k}}\right|.$ 
			Take $n$ sufficiently large so that $r_n>4$. Then  $\frac{r_k}{r} < \frac{r_k}{r_n ^2} < \frac{r_k}{4 r_n}$ which is less than $\frac{1}{4}$.

			For $k \leq n-1$, we have $r>r_k$ and therefore $I_k =\log \left(\frac{1+\frac{r}{r_k}}{-1+\frac{r}{r_k}}\right)  $ which can be rewritten as $\log \left(\frac{1+\frac{r_k}{r}}{1-\frac{r_k}{r}}\right)$. This quantity is seen to be less than $\frac{3 r_k}{r}$ using the following fact.
			
			\begin{equation}\label{Eq8}
				\log\left(\frac{1+x}{1-x}\right)<3x ~\mbox{for}~ 0<x<\frac{1}{2}.
			\end{equation} 
			Thus, $\sum_{k=1}^{n-1}I_k<3\sum_{k=1}^{n-1}\frac{r_k}{r}.$ Since $r> 4r_n$, $\sum_{k=1}^{n-1}I_k<3\sum_{k=1}^{n-1}\frac{r_k}{4r_n} < \frac{3 r_{n-1}}{4 r_n}\sum_{k=1}^{n-1}\frac{1}{2^{n-1-k}} $ (using Inequality~(\ref{growth-rn})). This is clearly less than $\frac{3 r_{n-1}}{2r_n}$.
			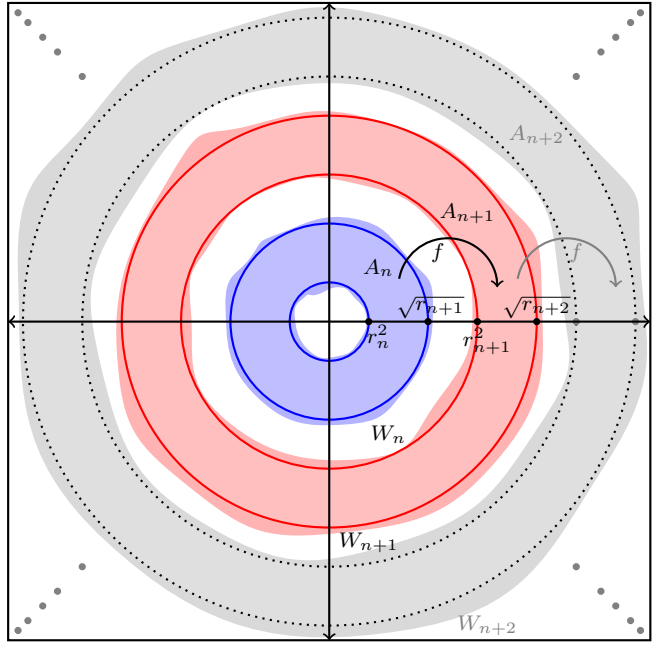
\begin{figure}
				\begin{center}
					
					\begin{tikzpicture}[scale=0.65]
						\fill[gray!30, even odd rule]
						plot[smooth cycle, tension=0.9] coordinates {
							(6.3,1.2) (5.3,3.6) (1.6,6.1)
							(-2.3,5.9) (-4.8,4.2)
							(-6.3,0.9) (-4.5,-4.5)
							(1.5,-6.3) (5.6,-2.9)
						}
					
						plot[smooth cycle, tension=0.9] coordinates {
							(4.7,0.5) (3.3,3.5) (-0.9,4.8)
							(-4.0,2.5) (-4.6,-1.0)
							(-2.4,-4.2) (1.4,-4.5) (4.2,-2.2)
						};
						
						\fill[red!30, even odd rule]
				 
						plot[smooth cycle, tension=0.9] coordinates {
							(4.3,0.6) (3.4,2.6) (1.0,4.1)
							(-1.6,4.0) (-3.2,3.0)
							(-4.3,-0.6) (-2.8,-3.4)
							(0.3,-4.3) (3.5,-2.7)
						}
					 
						plot[smooth cycle, tension=0.9] coordinates {
							(2.9,0.2) (2.2,1.9) (0.6,2.8)
							(-1.2,2.7) (-2.4,1.6)
							(-2.8,0.2) (-2.1,-2.0)
							(0.7,-2.8) (2.2,-1.6)
						};
						
						\fill[blue!30, even odd rule]
					 
						plot[smooth cycle, tension=0.9] coordinates {
							(2.1,0.3) (1.6,1.3) (0.5,2.1)
							(-1.0,1.8) (-1.6,1.3)
							(-2.1,0.3) (-1.5,-1.5)
							(0.3,-2.1) (1.6,-1.3)
						}
					 
						plot[smooth cycle, tension=0.9] coordinates {
							(0.75,0.1) (0.5,0.5) (0.2,0.7)
							(-0.25,0.55) (-0.5,0.5)
							(-0.7,0.1) (-0.4,-0.6)
							(0.2,-0.7) (0.6,-0.4)
						};

						\fill[gray!25, even odd rule]
						(0,0) circle (5)
						(0,0) circle (6.2);
						
						\draw[dotted, thick] (0,0) circle (5);
						\draw[dotted, thick] (0,0) circle (6.2);
						\fill[red!25, even odd rule]
						(0,0) circle (4.2)
						(0,0) circle (3.0);
						
						\draw[ red, thick] (0,0) circle (4.2);
						\draw[ red, thick] (0,0) circle (3.0);
						
						\fill[blue!25, even odd rule]
						(0,0) circle (0.8)
						(0,0) circle (2.0);
						
						\draw[ blue, thick] (0,0) circle (0.8);
						\draw[ blue, thick] (0,0) circle (2.0);
						\filldraw[color=black, fill=black, very thick](0.8,0) circle (0.04);
						\filldraw[color=black, fill=black, very thick](2.0,0) circle (0.04);
						\filldraw[color=black, fill=black, very thick](3.0,0) circle (0.04);
						\filldraw[color=black, fill=black, very thick](4.2,0) circle (0.04);
						
						\filldraw[color=gray, fill=gray, very thick](5.0,0) circle (0.04);
						\filldraw[color=gray, fill=gray, very thick](6.2,0) circle (0.04);
						\filldraw[color=gray, fill=gray, very thick](5.5,5.5) circle (0.04);
						\filldraw[color=gray, fill=gray, very thick](5.0,5.0) circle (0.04);
						\filldraw[color=gray, fill=gray, very thick](5.8,5.8) circle (0.04);
						\filldraw[color=gray, fill=gray, very thick](6.1,6.1) circle (0.04);
						\filldraw[color=gray, fill=gray, very thick](6.3,6.3) circle (0.04);
						\filldraw[color=gray, fill=gray, very thick](5.5,-5.5) circle (0.04);
						\filldraw[color=gray, fill=gray, very thick](5.0,-5.0) circle (0.04);
						\filldraw[color=gray, fill=gray, very thick](5.8,-5.8) circle (0.04);
						\filldraw[color=gray, fill=gray, very thick](6.1,-6.1) circle (0.04);
						\filldraw[color=gray, fill=gray, very thick](6.3,-6.3) circle (0.04);
						\filldraw[color=gray, fill=gray, very thick](-5.5,5.5) circle (0.04);
						\filldraw[color=gray, fill=gray, very thick](-5.0,5.0) circle (0.04);
						\filldraw[color=gray, fill=gray, very thick](-5.8,5.8) circle (0.04);
						\filldraw[color=gray, fill=gray, very thick](-6.1,6.1) circle (0.04);
						\filldraw[color=gray, fill=gray, very thick](-6.3,6.3) circle (0.04);
						\filldraw[color=gray, fill=gray, very thick](-5.5,-5.5) circle (0.04);
						\filldraw[color=gray, fill=gray, very thick](-5.0,-5.0) circle (0.04);
						\filldraw[color=gray, fill=gray, very thick](-5.8,-5.8) circle (0.04);
						\filldraw[color=gray, fill=gray, very thick](-6.1,-6.1) circle (0.04);
						\filldraw[color=gray, fill=gray, very thick](-6.3,-6.3) circle (0.04);
						\draw[thick, black, <-] (3.4,0.7) arc (0:170:1.0);
						\draw[thick, gray, <-] (5.8,0.7) arc (0:170:1.0);
						
						\node[black] at (2.2,1.4) {{\scriptsize $f$}};
						\node[gray] at (5.0,1.4) {\scriptsize $f$};
						\node[black] at (1.0,1.1) {\scriptsize $A_n$};
						\node[black] at (2.8,2.2) {\scriptsize $A_{n+1}$};
						\node[gray] at (4.2,3.8) {\scriptsize $A_{n+2}$};
						\node[black] at (1.0,-0.3) {\scriptsize $r_n^2$};
						\node[black] at (2.05,0.3) {\scriptsize $\sqrt{r_{n+1}}$};
						\node[black] at (3.2,-0.4) {\scriptsize ${r^2_{n+1}}$};
						\node[black] at (4.2,0.3) {\scriptsize $\sqrt{r_{n+2}}$};
						\node[black] at (1.2,-2.3) {\scriptsize $W_n$};
						\node[black] at (0.8,-4.5) {\scriptsize $W_{n+1}$};
						\node[gray] at (3.2,-6.2) {\scriptsize $W_{n+2}$};
						\draw[thick, <->] (-6.5,0) -- (6.5,0) node[right] {};
						\draw[thick, <->] (0,-6.5) -- (0,6.5) node[above] {};
						\draw[thick] (-6.5,-6.5) rectangle (6.5,6.5);
					\end{tikzpicture}
					\caption{{\label{bov}
							Iterated images of $A_n$ under $f$. }}
				\end{center}
			\end{figure}
			\par For $k\geq n+2, r< \frac{r_{n+1}}{4}$ and $r_k \geq r_{n+2}$, which gives that $\frac{r}{r_k} < \frac{r_{n+1}}{4 r_{n+2}}<\frac{1}{4}$. Using this and  Inequality~(\ref{Eq8}), we have
			$ \log \left( \frac{1+\frac{r}{r_k}}{ 1-\frac{r}{r_k}}\right)  <   \frac{3r}{r_k} < \frac{3 r_{n+1}}{4r_{k}}$. Now
			$\sum_{k=n+2}^{\infty}I_k<\sum_{k=n+2}^{\infty} \frac{3r_{n+1}}{4r_k}=\frac{3r_{n+1}}{4r_{n+2}}\sum_{k=n+2}^{\infty} \frac{r_{n+2}}{r_k}$. It now follows from Inequality~(\ref{growth-rn}) that    $\sum_{k=n+2}^{\infty}I_k < \frac{3r_{n+1}}{4r_{n+2}} \sum_{k=n+2}^{\infty} \frac{1}{2^{k-n-2}}$. This is less than 
			$\frac{3r_{n+1}}{2r_{n+2}} $.
			
			\par Putting all these estimates together, we have 	
			\begin{align*}
				\log\left|\frac{f(r)}{f(-r)}\right|&< \frac{3r_{n-1}}{2r_n} +I_n +I_{n+1}+ \frac{3r_{n+1}}{2r_{n+2}}.
			\end{align*}
			Since $\frac{r_n}{r}, \frac{r}{r_{n+1}} < \frac{1}{4}$ and $\log \left(\frac{1+x}{1-x}\right)$ is increasing in $(0,1)$, we have	
			$\log\left(\frac{{1+\frac{r}{r_{n+1}}}}{{1-\frac{r}{r_{n+1}}}}\right)+  	\log\left(\frac{{1+\frac{r_n}{r}}}{{1-\frac{r_n}{r}}}\right) <   2 \log\left(\frac{1+\frac{1}{4}}{{1-\frac{1}{4}}} \right)=2 \log \frac{5}{3}< \log 4$. 
			Since $\lim_{n\to\infty}\frac{r_{n}}{r_{n+1}} =0$, we have $\log \left| \frac{f(r)}{f(-r)}\right|<  \log4$ whenever $z \in A_n $ for all sufficiently large $n$.
			Therefore, $$|f(z)|\geq f(-|z|)>\frac{1}{4}f(|z|)>\frac{1}{4}f({r_n^2})>{r^{2}_{n+1}} ~\mbox{for all sufficiently large}~ n.$$  The last inequality follows from Property 4.  
			\par
			It follows from Lemma~\ref{BWD-lem} that the annulus $A_n $ is contained in a BWD for all sufficiently large $n$. A schematic picture of $A_n$ is given in Figure~\ref{bov}.
		\end{proof}
		For using later, we make a remark at this point.
		\begin{Remark}
			Recall from Inequality~(\ref{quotient-rn}) that $ \frac{r_{n+1}}{r_n ^4}> \frac{2C r_{n}^{n-3}}{(r_{n-1})^{n-1}}$. This gives that $ \frac{r_{n+1}}{r_n ^m }> \frac{2C r_{n}^{n-m+1}}{ (r_{n-1})^{n-1}}$. Using $r_{n}> Cr^2_{n-1}$, we have  $ \frac{r_{n+1}}{r_n ^m }>  2C^{n-m+2} r_{n-1}^{n-2m+3} = 2 C^{m-1}(Cr_{n-1})^{n-2m+3} $ and this goes to $\infty$ as $n \to \infty$. Therefore, for every fixed $m$, we have $\lim_{n\to\infty} \frac{r_{n+1}}{r^m_n} = \infty$. Let $M>1$ be arbitrary  and choose $m =[M]+1$ where $[.]$ denotes the greatest integer function. Then, for all sufficiently large $n$, $\frac{r_{n+1}}{r_n ^{m}}>M$, and this gives that $\frac{\log r_{n+1}}{\log r_n} > \frac{\log M}{\log r_n} +m > m >M$.  In other words, \begin{equation}
				\lim_{n\to\infty} \frac{\log r_{n+1}}{\log r_n} =\infty.
				\label{logrn-limit} \end{equation} 
		\end{Remark}
		The presence of critical points in a BWD is crucially related to its connectivity, and that is discussed in Section~\ref{singularvalues-section}.
		It was not Baker but later researchers who showed that the BWDs discussed in Example~\ref{1st exm} contain critical points. However, modifying his own example, Baker also established the existence of critical points in BWD and that is the next stuff for discussion. 
		
		\begin{example}(\cite{Baker 1985})
			There is an entire function with an infinitely connected BWD. \label{2nd exm}
		\end{example}
		\begin{proof}
			
			Let $0<C<\frac{1}{4e^2}$ and $r_1>1$ be fixed. Then choose a positive integer $n_0$ such that $2^{n_0-1}C>2r_1$. For each $n\leq n_0$, choose $r_{n+1}>2r_n$ and for $n>n_0$, define 
			\begin{equation}\label{Eq9}
				r_{n+1}=C^2\prod_{i=1}^{n}\left(1+\frac{r_n}{r_i}\right)^2.
			\end{equation} For $n>n_0$, we are going to demonstrate that $r_n$ increases in a rate depending on the square of $r_{n-1}$. Keeping the first term intact, observing that all other terms in $r_{n_0 +2}$ are bigger than $2$, and then using the choice of $C$, we see that 
			\begin{align*}
				r_{n_0+2}=C^2\prod_{i=1}^{n_0+1}\left(1+\frac{r_{n_0+1}}{r_i}\right)^2 >C^22^{2n_0} \left(1+\frac{r_{n_0+1}}{r_1}\right)^2 
				&>C^22^{2n_0} \frac{r^2_{n_0+1}}{r^2_1} 
				>{ 4}^2 r^2_{n_0+1}.\\
			\end{align*} 
			%
			Inductively,  we have \begin{equation}r_{n_0+k+1}>{ 4 }^{k+1} r^2_{n_0+k}~\mbox{for every natural number }~k.\label{rn}\end{equation}
			Consider  the function
			\begin{equation}\label{Eq10}
				g(z)=C^2\prod_{i=1}^{\infty}\left(1+\frac{z}{r_i}\right)^2.
			\end{equation} Since $r_{i+1}>2r_i$ for all $i$,  we have $\sum_{i=1}^{\infty}\frac{1}{r_i}<\frac{2}{r_1} $.  Therefore, the infinite product in (\ref{Eq10}) converges on every compact subset of $\mathbb{C}$ and $f$ is an entire function. 
			\par Since $\frac{1}{4^{n-n_0+1}}\to 0$ and $\frac{n+1}{n+2}\to 1$ as $n\to \infty$, putting $k =n-n_0$ in ~(\ref{rn}),  we have $$r^2_n<\frac{r_{n+1}}{4^{n-n_0+1}}<r_{n+1}\frac{n+1}{n+2}~\mbox{ for all sufficiently large  } n.$$
			
			Take $s_n=(\frac{n+1}{n+2})r_{n+1}$ and consider the annulus $A_n=\{z:r_n^2<|z|<s_n\}$. For all sufficiently large $n$, one can check that $g(s_{n})< s_{n+1}$, $g(-s_{n})>r^2_{n+1}$ and $g(-r^2_n)>r^2_{n+1}$ (see~\cite{Baker 1985}). For $|z|=s_n$, $|g(z)|\leq g(|z|)=g(s_n)<s_{n+1}$. Now by the Maximum Modulus Theorem, $|g(z)|<s_{n+1}$  for all $|z|<s_n$. Moreover, for all $z$,
			$$|g(z)|=C^2\prod_{i=1}^{\infty}\left|1+\frac{z}{r_i}\right|^2\geq C^2\prod_{i=1}^{\infty}\left(1+\frac{-|z|}{r_i}\right)^2=g(-|z|).$$ This gives for $|z|=r_n^2$ that $|g(z)|\geq g(-r_n^2)>r_{n+1}^2$. By the Minimum Modulus Theorem, $|g(z)|>r_{n+1}^2$ for all $z$ with $|z|>r_n^2$.
			Thus,  the function $g$ maps $A_n$ into  $A_{n+1}$ for all sufficiently large $n$. The annulus $A_n$ is contained in a BWD of $g$ by Lemma~\ref{BWD-lem}.
			\par In order to locate the critical points of $g$, note that			
			\begin{equation}\label{Eq11}
				\log g(z)=\log C^2+2\sum_{i=1}^{\infty}\log \left(1+\frac{z}{r_i}\right).
			\end{equation}
			We consider this expression for those $z$ for which  $\log g(z) $ is defined. Then
			$${\frac{1}{2}g'(z)}={g(z)}\sum_{i=1}^{\infty} {\frac{1}{z+r_i}}.$$
			Let $\phi(z)= \sum_{i=1}^{\infty} {\frac{1}{z+r_i}}$.
			The zeros of $g'$ are precisely the zeros of $g$ and those of $\phi$. 
			Fix a sufficiently large $n$ and note that $\phi$ is strictly decreasing on
			$I_n=(-r_{n+1}, -r_n)$ and $-r_{n+1} < -s_n <-r_n ^2 < -r_n $.
			%
			We assert that $\phi(-s_{n})>0$ and $\phi(-r^2_n)<0$. Let   $ \frac{1}{r_{i }-s_{n}}$ be denoted by $\alpha_{i}$. Since for sufficiently large $n$, $r_{n+1}>2^{n+1}r_n>(n+2)r_n$ we have $(n+1)r_{n+1} -(n+2)r_n > n r_{n+1}$ and consequently,
			
			\begin{equation} \alpha_{n}=\frac{-(n+2)}{ (n+1)r_{n+1}-(n+2)r_n}> \frac{-(n+2)}{nr_{n+1}}.
				\label{ alpha-n}\end{equation}
			Since all the terms after the $(n+1)$-th term in the series   expression of $\phi (-s_n)$ are positive, $\phi(-s_n)>\left( \sum_{i<{n+1}}\frac{1}{r_{i}-s_{n}}\right)+\frac{n+2}{r_{n+1}}$. For $i<n$ we have 	$ \alpha_{i } >  \alpha_{n}$ which gives that $\phi(-s_n)> n \alpha_n + \frac{n+2}{r_{n+1}}.$ Now, it follows from Inequality~(\ref{ alpha-n}) that $\phi(-s_n)>0$.
			
			Since $r_{n+1}>4r^2_n$ for all sufficiently large $n$ (by Inequality~(\ref{rn})), we have  $r_{n+1}-r^2_n>3r^2_n$. Repeated application of this inequality gives that $r_{n+2}-r^2_n>7r^2_n$, $r_{n+3}-r^2_n>15r^2_n$ and so on. Since $\sum_{n=2}^{\infty}\frac{1}{2^n-1}<\sum_{n=2}^{\infty}\frac{1}{2^n}=1$, we have 
			\begin{align*}
				\phi(-r^2_n)&<\frac{1}{r_n-r^2_n}+\left(\frac{1}{r_{n+1}-r^2_n}+\frac{1}{r_{n+2}-r^2_n}+\cdots\right)\\
				&<-\frac{1}{r^2_n}+\left(\frac{1}{3r^2_n}+\frac{1}{7r^2_n}+\frac{1}{15r^2_n}+\cdots\right)\\
				&=-\frac{1}{r^2_n}+\frac{1}{r^2_n}\sum_{n=2}^{\infty}\frac{1}{2^n-1}<0. 
			\end{align*} 
		
			Thus, the function $g$ has a critical point in $A_n$, and therefore the BWD containing the annulus $A_n$ contains a critical point, namely a root of $\phi$. The proof of the fact that these BWDs are infinitely connected is postponed to Corollary~\ref{SecondExample-infinitelyconnected}. 
		\end{proof}
		The roots of $g$ are all negative and are critical points of $g$. Since $g(0)=C^2$, all the roots of $g$ have the same forward orbit. 
		Further, all but possibly finitely many roots of $\phi$ (these are critical points of $g$) are in the BWDs discussed in the above example. By studying the forward orbit of these finitely many critical points as well as of $0$, one may understand  the Fatou set of $g$ completely. Below and now onward, by the grand orbit of a wandering domain $W$ of $f$, we mean the set of all Fatou components $W$ such that $W_m =W_n$ for some non-negative integers $m,n$.
		A question arises.
		\begin{Question}
			Is there a Fatou component of $g$ that is not in the grand orbit of the BWDs constructed in Example~\ref{2nd exm}? If so, then what is the nature of all such Fatou components? 
		\end{Question}
		It is worth mentioning in the context of the above question that there is an entire function for which a simply connected bounded wandering domain can actually exist along with a BWD (see Theorem 1, \cite{Berg 2010-simply-multiply-wd}).
		
		Recall that a maximally connected subset of the Julia set is called a Julia component.
		A Julia component is called buried if it does not intersect the boundary of any Fatou component.
		\begin{Remark}
			Baker and Dom\'inguez  constructed an entire function similar to Example~\ref{2nd exm} with BWD with the additional property that every repelling periodic point of the function is a buried and singleton Julia component (Theorem G, \cite{Baker Dominguez 2000}).  
		\end{Remark}
		The order of a transcendental function $f$ is a quantification of the rate of growth of its maximum modulus, which  is defined by 
		$$\label{For1}
		\limsup_{r\to\infty}\frac{\log(\log (M(r,f)))}{\log r} 
		$$
		where $M(r,f)=\max_{|z|=r}|f(z)|$  (see \cite{asymptoticvalue-book}). 
		It was proved in Theorem $5.2$ in \cite{Baker 1984} that for every $\rho$ satisfying $1\leq\rho\leq\infty$, there is an entire function of order $\rho$ having a wandering domain. The wandering domains constructed in this proof are simply connected. Later in 1985, Baker proved the following result on the existence of a BWD of an entire function of any prescribed order.
		
		\begin{theorem}(Theorem 1, \cite{Baker 1985})
			For each $\rho \in [0, \infty)$, there is an entire function of order $\rho$, which has a BWD.
			\label{exampleBWD-order}
		\end{theorem} 
		\begin{proof}
			Let $0 \leq \rho < \infty$. Also, let $C'$ be a constant with $0<C'<\frac{1}{4e^2}$, $r_1>1$ and  $k_1 \geq 1$ be a natural number.  Take a natural number $n_0$ such that $2^{n_0-1}C'>2r^{k_1}_1$. Consider the sequence $\{r_n\}_{n\geq 1}$ such that $r_{n+1} >  2r_n$ for $1\leq n <n_0$ and for $ n \geq n_0$, consider 
			\begin{equation}
				r_{n+1}=C'\prod_{j=1}^{n}\left\{1+\left(\frac{r_n}{r_j}\right)^{k_j}\right\},~\mbox{ where}~ k_n~\mbox{ is the greatest integer not exceeding }~r_n ^\rho.
				\label{rn-oredr-example} 
			\end{equation}In order to define $r_{n+1}$, one requires $r_1, r_2,\cdots r_n, k_1, k_2, \cdots k_{n-1}$ and $k_n$,  all of which are already defined in the previous step. By induction, one can easily see that for $n\geq n_0$, $r_{n+1}>2r_n$. This gives that  \begin{equation}
				r_{n+k}>2^{k} r_n ~\mbox{for each }~n~\mbox{and}~ k.
				\label{r-n-growth}
			\end{equation} 
			
			The choice of $k_n$s is made in such a way that the order of the following function becomes $\rho$ while still exhibiting BWD.
			\par Consider the function 
			$$g(z)=C'\prod_{j=1}^{\infty}\left\{1+\left(\frac{z}{r_j}\right)^{k_j}\right\}.$$
			For $\rho =0$, each $k_j$ is  $1$. Then $g(z)=\frac{C' f(z)}{C z^2}$ where $f$ is the function given by ~(\ref{Baker-first}). The function $g$ is clearly entire.
			Let $0<\rho <\infty$. Then it follows from Inequality~(\ref{r-n-growth}) that $k_j >j$ for all large $j$. For each compact subset $K$ of the complex plane, there exists $M>0$ such that   $\sum_{j=1}^{\infty}\left|\frac{z}{r_j}\right|^{k_j}\leq \sum_{j=1}^{\infty}\left(\frac{M}{r_j}\right)^{k_j}<\infty$ for all $z \in K$. There exists $N$ such that $\frac{M}{r_j}<\frac{1}{2}$ for all $j\geq N$ and consequently, $\sum_{j=1}^{\infty}\left|\frac{z}{r_j}\right|^{k_j}= M_1 +\sum_{j=N}^{\infty}\left(\frac{1}{2}\right)^{k_j}$ for some $M_1>0$. Clearly, this series converges as $k_j >j$ for all large $j$. Thus $g$ represents an entire function.
			\par Recall that, in Example~\ref{1st exm}, the Inequality~(\ref{f(r)-inequality})  is a consequence of  Hadamard's Three Circles Theorem (true for all entire functions) and the three properties of $f$ namely, Inequality~(\ref{f(rn)-bound}) i.e., $r_{n+1} < f(r_n) <er_{n+1}$ where $r_n$ is the modulus of a root of $f$, as defined in Example~\ref{1st exm}, $M(r,f)=f(r)$ and $|f(z)| < \frac{1}{4}$ for all $|f(z)|\leq 1$. We are going to show that $g$ satisfies all these conditions. Taking $r_n$ as defined in Equation~(\ref{rn-oredr-example}), we observe that $$r_{n+1} =C' \prod_{j=1}^{n} \left\{1+\left(\frac{r_n}{r_j}\right)^{k_j}\right\} < g(r_n)=r_{n+1} \prod_{j=n+1}^{\infty} \left\{1+\left(\frac{r_n}{r_j}\right)^{k_j}\right\}.$$
			It follows from Inequality(\ref{r-n-growth}) and $k_j >1$ that $(\frac{r_n}{r_j})^{k_j} <\frac{r_n}{r_j} <  \frac{1}{2^{(j-n)}}$ for all $j \geq n+1$. The rightmost term in the above inequality now becomes less than $r_{n+1} \prod_{j=1}^{\infty} \left\{1+ \frac{1}{2^j} \right\} $, which is nothing but $e r_{n+1}$. Therefore, \begin{equation}
				r_{n+1}< g(r_n)< e r_{n+1}.
				\label{g-rn}
			\end{equation}
			$|g(z)| \leq g(|z|)$ for all $z$, we have $M(r,g)=g(r)$. Now $g(1)=C' \prod_{j=1}^{\infty}\left\{1+\left(\frac{1}{r_j}\right)^{k_j}\right\}  \leq C'\prod_{j=1}^{\infty} \left\{1+ \frac{1}{r_j}  \right\}=\frac{C'}{C} f(1)$, which is less that $\frac{1}{4}$ by the choice of $C,C'$ and the fact that $f(1) < \frac{1}{4}$. Therefore $$g(r^2) > 4{g(r)}^2.$$  Putting $r=\sqrt{r_{n}}$ and $r=r_n$  and using Inequality~(\ref{g-rn}), we get that for $n>n_0$, $g(\sqrt{r_n})<\sqrt{r_ {n+1} }$ and $g(r_n^2)>4 r_{n+1}^2$, respectively. 
			\par Consider the annulus $A_n=\{z:r_n^{2}\leq|z|\leq \sqrt{r_{n+1}}\}$. For $z\in A_n$, the Maximum Modulus Theorem and $g(\sqrt{r_{n}})<\sqrt{r_{n+1}}$ give that \begin{equation}
				|g(z)|<g(\sqrt{r_{n+1}})<\sqrt{r_{n+2}}.
				\label{g-Annulus}
			\end{equation} 
			
			To obtain a lower bound for $|g(z)|$ on $A_n$, we put  $|z|=r$ and observe that $r<r_j$ if and only if $j >n$. This gives that $$\frac{g(r)}{|g(z)|}
			\leq \prod_{j>n}\left(\frac{{1+\frac{r}{r_j}}}{{1-\frac{r}{r_j}}}\right) \prod_{j\leq n}\left(\frac{{1+\frac{r}{r_j}}}{{ \frac{r}{r_j} -1}}\right) = \prod_{j>n}\left(\frac{{1+\frac{r}{r_j}}}{{1-\frac{r}{r_j}}}\right) \prod_{j\leq n}\left(\frac{{1+\frac{r_j}{r}}}{{1-\frac{r_j}{r}}}\right).$$ 
			For all sufficiently large $n$, 
			\begin{equation}		 
				\frac{r}{r_j}\leq \frac{r}{r_{n+1}}<\frac{1}{\sqrt{r_{n+1}}} < \frac{1}{2} ~\mbox{for}~ j>n ~\mbox{and}~\frac{r_j}{r}\leq \frac{r_n}{r_n ^2}=\frac{1}{ r_{n}}< \frac{1}{2} ~\mbox{for}~ j\leq n.
				\label{Ineq-order-example}\end{equation}
			Using Inequality~(\ref{Eq8}), we have
			$$\log\left(\frac{{1+\frac{r}{r_j}}}{{1-\frac{r}{r_j}}}\right)<\frac{3r}{r_j} ~~\mbox{for}~ j>n~\mbox{and}~ \log\left(\frac{{1+\frac{r_j}{r}}}{{1-\frac{r_j}{r}}}\right)<\frac{3r_j}{r} ~\mbox{for}~ j\leq n.$$ Therefore,  $$\log \frac{g(r)}{|g(z)|}\leq 3\sum_{j>n}\frac{r}{r_j}+3\sum_{j\leq n}\frac{r_j}{r}<\frac{3r}{r_{n+1}}\sum_{j \geq n+1} 2^{n+1-j}+ \frac{3r_n}{r} \sum_{j \leq n } 2^{-n+j}< 6\left(\frac{r}{r_{n+1}}+\frac{r_n}{r}\right).$$
			It follows from Inequality~(\ref{Ineq-order-example}) that the right hand side term is less than 
			$\frac{1}{\sqrt{r_{n+1}}}+ \frac{1}{r_n}$, which is less than $\log 4$
			for all sufficiently large $n$.	  In other words,  $|g(z)|>\frac{1}{4}g(|z|)$ for all $z \in A_n$ and for all sufficiently large  $ n$. Using this and the Minimum Modulus Theorem we get $|g(z)|> \frac{1}{4}g(r_n ^2)$ for all $z \in A_n$. Since  it is already observed that $g(r_n^2)>4 r_{n+1}^2  >r^2_{n+1}$, using Inequality~(\ref{g-Annulus}), we have $g(A_n)\subset A_{n+1}$. By Lemma~(\ref{BWD-lem}),   the annulus $A_n$ is contained in a BWD of $g$. Note that all these calculations are valid for $k_j =1$ (when $\rho =0$) as well as for $k_j >1$(when $0< \rho <\infty$).
			
			\par To determine the order of $g$, we take the logarithm of the maximum modulus of $g$ on $|z|=r$  where $r_n\leq r<r_{n+1}$, i.e.,  $$\log g(r)=\log C+\sum_{j\leq n}\left[\log\left\{ 1+ \left(\frac{r}{r_j}\right)^{k_j}\right\}\right]+\sum_{j> n}\left[\log\left\{ 1+ \left(\frac{r}{r_j}\right)^{k_j}\right\}\right].$$
			Rewriting the second term, we have
			\begin{equation}\label{Eq12}
				\log g(r)=\log C+\sum_{j\leq n}k_j\log\left(\frac{r}{r_j}\right)+\sum_{j\leq n}\log\left\{ 1+ \left(\frac{r_j}{r}\right)^{k_j}\right\}+\sum_{j>n}\log\left\{ 1+ \left(\frac{r}{r_j}\right)^{k_j}\right\}.
			\end{equation} 
			For $j\leq n$, $r_j \leq r_n  \leq r$ and therefore $0<\left(\frac{r_j}{r}\right)^{k_j}<1$. This gives (using $\log (1+x) <x$ for all $x \in (0,1)$) that $$\sum_{j\leq n}\log\left\{ 1+ \left(\frac{r_j}{r}\right)^{k_j}\right\}\leq \sum_{j\leq n}\left(\frac{r_j}{r}\right)^{k_j}\leq \sum_{j \leq n } \left(\frac{r_j}{r_n}\right)^{k_j} \leq \sum_{j \leq n } \frac{1}{2^{n-j}} < 1+ \sum_{j =1}^{\infty} \frac{1}{2^{ j}}=2.$$ 
			
			Similarly, for $j> n$, $r_j> r$, we have $0<\left(\frac{r}{r_j}\right)^{k_j}<1$ and, $$\sum_{j> n}\log\left\{ 1+ \left(\frac{r}{r_j}\right)^{k_j}\right\}\leq \sum_{j> n}\left(\frac{r}{r_j}\right)^{k_j} < \sum_{j> n }\left(\frac{r_{n+1}}{r_j}\right)^{k_j}< \sum_{j>n} 2^{n+1-j}=2 .$$
			Therefore, the sum of the third and the fourth term in the right hand side of the Equation (\ref{Eq12}) is less than $4$ which is independent of $r$. Now, Equation (\ref{Eq12}) can be written as,
			\begin{equation}
				\log g(r)=\sum_{j\leq n}\left\{k_j\log\left(\frac{r}{r_j}\right)\right\}+\mathcal{O}(1).
				\label{log-g(r)}
			\end{equation}	
			where $\mathcal{O}(1)$ is a function of $r$ that is bounded (as $r\to\infty$). We now estimate the sum using a suitable integral.
			\par The circle $\{z: |z|=r_j\}$ contains exactly $k_j$ many zeros of $g$ counting multiplicity.	
			Let $n(t)$ denote the number of zeros of $g$ in $\{z:|z|\leq t\}$. Then $n(t)=\sum_{r_j \leq t} k_j$ and is a non-decreasing step function. Since $n(t)=0$ for all $t\in (0, r_1)$, the Riemann-Stieltjes integral
			$\int_{0}^{r}\log\left(\frac{r}{t}\right)d(n(t)) $ is equal to $\int_{\delta}^{r}  \log\left(\frac{r}{t}\right)d(n(t))$ for some $0< \delta < r_1$. Observe that the value of this integral is $\sum_{j\leq n} k_j\log\left( \frac{r}{r_j}\right)$. On the other hand, integrating by parts, the value of this integral is found to be $\int_{\delta}^{r} \frac{n(t)}{t} dt$. Thus \begin{equation}
				\sum_{j\leq n} k_j\log\left( \frac{r}{r_j}\right) = \int_{\delta}^{r} \frac{n(t)}{t} dt.
				\label{sum-integral}
			\end{equation}
			\par For every $t \in (\delta, r)$, there is a $j$ such that $r_j \leq t < r_{j+1}$. Then $n(t)=k_1+k_2+\cdots+k_j  \leq r_1 ^{\rho} + r_2 ^{\rho}+\cdots  +r_j ^{\rho}$ by our choice of $k_i$ (see Equation~\ref{rn-oredr-example}). Since $r_i < \frac{r_j}{2^{j-i}}$ for each $i<j <n$ (see Inequality~(\ref{r-n-growth})), $n(t) \leq r_j ^\rho \sum_{i=1}^{j} 2^{(i-j)\rho}$ which is clearly less than $r_j ^\rho \frac{2^\rho}{2^{\rho}-1}$. Thus (using $r_j \leq t$), we have
			$$\int_{0}^{r}\frac{n(t)}{t}dt=\int_{\delta}^{r}\frac{n(t)}{t}dt \leq\frac{2^\rho}{2^{\rho}-1}  \int_{\delta}^{r}t^{\rho -1}dt = \frac{2^\rho}{\rho(2^{\rho}-1)} (r^\rho -\delta^\rho)     .$$ 
		
			Finally, using Equations~(\ref{log-g(r)}) and (\ref{sum-integral}) we get $\log g(r)=\mathcal{O}(r^\rho)$ as $r\to\infty$.  This shows that $g$ is an entire function of order $\rho$.	 The case $\rho=0$ is almost obvious since the last inequality tends to $\log r - \log \delta$ as $\rho \to 0$.
		\end{proof} 
		\begin{Remark}
			An entire function of order $\infty$ can also be constructed which has a BWD  \cite{Baker 1985}. 
		\end{Remark}
		
		Motivated by the work of Baker, Bergweiler et al., in 2013, constructed some  examples of entire functions having Baker wandering domains that differ significantly from the earlier ones  \cite{BergRippStall 2013}. Earlier in 2008, Kisaka and Shishikura, using quasiconformal surgery, had also provided some examples of entire functions possessing Baker wandering domains \cite{KisakaShishikura 2008}. An example of such a function whose Julia set has Hausdorff dimension $1$ is constructed by Bishop~\cite{Bishop-2018}. This is a surprising as well as an important result on transcendental Julia sets. 
		\par Two entire functions $f$ and $g$ are called permutable if they permute, i.e., $f\circ g=g\circ f$. In 2015, Benini et al. investigated the problem of equality of Julia sets for two permutable functions by analyzing the limits of the iterates on their BWDs \cite{Benini rippon Stallard 2015}. We put together two results reported in that paper.
		
		\begin{theorem} (Proposition 3.3 and Corollary 3.6)
			Let $f$ and $g$ be two permutable entire functions and $U$ be a BWD of $f$. Then $g(U)$ is contained in a BWD of $f$. Further,  $f^n(U)$ is a BWD of   $g$ for all large $n$.
			\label{Benini-BWD}	\end{theorem}
		Theorem~\ref{Benini-BWD} can be used to get a BWD of an entire function from a BWD of another function whenever the two functions commute. In this context, the following question arises.	
		\begin{Question}
			Does there exist an entire function which commutes with the functions known to have BWD such as Examples~\ref{1st exm} and \ref{2nd exm} or the example appearing in Theorem~\ref{exampleBWD-order}?
		\end{Question}
		\subsection{Meromorphic functions}  \label{meromorphic}
		In the presence of a BWD, all the Julia components (except the one containing $\infty$) are bounded. 
		A single bounded Julia component of an entire function implies the existence of a BWD  (by Corollary~\ref{BWD-iff-MCFC}). This is not  true in general for meromorphic functions due to the possible presence of a pole. For example, there are functions with infinitely many poles having a bounded Julia component but no BWD (see ~\cite{NayakZheng 2011}). However, as shown by Zheng in 2002, boundedness of all Julia components gives rise to a BWD for every meromorphic function with finitely many poles. This class of functions is a natural generalization of entire functions.  
		\begin{theorem}(Theorem 1, \cite{Zheng 2002})
			Let $f$ be a meromorphic function with at most finitely many poles, then $f$ has a BWD if and only if all the Julia components (except the one containing $\infty$) of $f$ are bounded.
			\label{FP-BWD-IFF}
		\end{theorem}
		Though it is not explicitly mentioned in the paper, BWD for meromorphic functions with poles appeared, probably for the first time, in the work of Baker and co-authors (refer to $f_3$ on page 268 in \cite{bk2}). 
		Further such examples were reported by Rippon and Stallard in 2005 (see \cite{RipponStallard 2005(1)}). Their idea of construction was similar to Baker's first example and involves the outer sequence of a meromorphic function  with finitely many poles. An outer set corresponding to a Jordan curve $\gamma$  that winds around $0$  is the closure of the unbounded component of the complement of $\gamma$. An outer sequence for $f$ is a sequence $E_n$ of outer sets corresponding to $\gamma_n$ such that 
		\begin{enumerate}
			\item  there is some circle $\{z:|z|=R\}$ surrounding the poles of $f$ and surrounded by $\gamma_n$ for all $n$;
			\item  the Euclidean distance between $\gamma_n$ and $0$ tends to infinity as $n\to\infty$;
			\item $\gamma_{n+1}\subset f(\gamma_n)$,  and  every component of $f^{-1}(E_{n+1})$ lies in $E_n$ or in $\{z:|z|\leq R\}$ for each $n$.
		\end{enumerate}
		Define
		$$B(f)=\{z:~\mbox{there exists}~ L ~\mbox{such that}~ f^{n+L}(z)\in E_n ~\mbox{for all}~ n\}.$$
		
		For a meromorphic function $f$ with finitely many poles, there exists an outer sequence for $f$, and $B(f)$ does not depend on any particular choice of the outer sequence. Also, $B(f)$ is always non-empty and is completely invariant under $f$. Moreover, the boundary $\partial B(f) $ of $B(f)$ is the same as $\mathcal{J}(f)$ and $B(f)\cap\mathcal{J}(f)\neq \emptyset$. These results and their proofs can be found in \cite{RipponStallard 2005(1)}. We need the following result in order to explain the next example.
		\begin{theorem}(Theorem 4 and Theorem 2(a), \cite{RipponStallard 2005(1)})\label{BF}
			Let $f$ be a meromorphic function with finitely many poles. Then,
			\begin{enumerate}
				\item There exists an unbounded closed connected subset $\Gamma$ of $B(f)$ such that for every $z_0\in B(f)$, $f^{n_0}(z_0)\in \Gamma$ for some $n_0$.
				\item If $U$ is a Fatou component intersecting $B(f)$, then $U\subset B(f)$ and $U$ is a wandering domain.
			\end{enumerate}   
		\end{theorem}
		
		\begin{example}(\cite{RipponStallard 2005(1)})\label{Exm 3}
			If $f$ is as given in (\ref{Baker-first}) then  $F(z)=\left(\frac{z}{z-a}\right) f(z)$ has a BWD for every $a \in (0, \frac{1}{6})$.
		\end{example}
		\begin{proof}
			Consider $F(z)=\left(\frac{z}{z-a}\right) f(z)$ where $f$ is the entire function discussed in Example \ref{1st exm}. Let $0<a<\frac{1}{6}$. Then, for $|z| \leq 2a < \frac{1}{3}$, we have by Property 1 in Example~\ref{1st exm} that     $|f(z)| < \frac{1}{4} |z|^2 < \frac{1}{12}|z|$.   Now  for $|z-a|\geq\frac{a}{2}$, we have $|\frac{z}{z-a}|\leq 1+|\frac{a}{z-a}|\leq 3$. Then $|f(z)| < \frac{1}{4}|z|$ for all $z$ in the doubly connected region $S= \{z: |z-a|\geq\frac{a}{2} ~\mbox{and}~ |z|\leq 2a \}$. Since $f$ has a super-attracting fixed point at $0$, so does $F$. It now follows that the set $S$ is contained in the immediate basin of attraction of $0$, i.e., the component of the attracting basin of $0$ containing $0$. This immediate basin of attraction surrounds a pole. \par Let $A_n=\{z:r_n^2<|z|<\sqrt{r_{n+1}}\}$ and $A'_n=\{z:(1+\epsilon)r_n^2<|z|<(1-\epsilon)\sqrt{r_{n+1}}\}$ for some $\epsilon >0$. This $\epsilon$ can be seen to be independent of $n$. From Example \ref{1st exm}, it follows that $f(A_n)\subset A'_{n+1}$  for all large $n$. 
			
			Since $\lim_{z \to \infty}
			\frac{F(z)}{f(z)} = 1$, there is $M>0$ such that for $|z|>M,$ we have $ |F(z)|- |f(z)| \leq |F(z)- f(z)|< \epsilon |f(z)|$. This gives that $
			|F(z)| < (1+\epsilon) |f(z)|< (1-\epsilon^2) \sqrt{r_{n+2} }<\sqrt{r_{n+2} }$ for $z \in A_n '$. Considering 
			$\lim_{z \to \infty}
			\frac{f(z)}{F(z)} = 1$ and arguing similarly we get $|f(z)|-|F(z)| \leq |F(z)-f(z)|< \epsilon |F(z)|$ for $|z|>M$. This gives that $(1+\epsilon)|F(z)| >|f(z)|> (1+\epsilon) r_{n+1}^2$. In other words, $|F(z)| > r_{n+1}^2$. Therefore $F(A_n) \subset A_{n+1}$  for all sufficiently large $n$. Therefore, $A_n$ lies in some multiply connected Fatou component of $F$, say $U_n$. By Theorem \ref{BF}(1), $U_n$ intersects $B(F)$ for all sufficiently large $n$, and by Theorem \ref{BF}(2), $U_n$ is wandering. Clearly, all $U_n$s are distinct, $U_{n+1}$ surrounds $U_{n}$ for all large $n$, and hence $U_n$ is a BWD of $F$. 
		\end{proof}	
		In 2005, Rippon and Stallard also provided some criteria for the existence of a BWD for a meromorphic function with finitely many poles. 
		\begin{theorem}(Theorem 3, \cite{RipponStallard 2005(1)})\label{Thm2.7}
			Let $f$ be a meromorphic function with a finite number of poles. There exists $r>0$ such that if $U$ is a Fatou component surrounding the disk $\{z:|z|\leq r\}$ then $U$ is a BWD of $f$.	
		\end{theorem}
		In the above theorem,   $r=\max\{R, \sup_{|z|=R} |f(z)|\}$ where $R=\max\{|z|: f(z)=\infty\}$.

		A multiply connected Fatou component of a meromorphic function with infinitely many poles is not necessarily BWD. For example, $z+2+e^{-z}+\frac{1}{100(z-1-i\pi)}$ has a multiply connected invariant Fatou component (see Example $1$ in \cite{Dominguez 1998}). Another example is the function $\lambda\sin z-\frac{\epsilon}{z-\pi}$ where $0<\lambda<1$ and $\epsilon>0$. For sufficiently small $\epsilon$, the Fatou set is a single completely invariant domain of infinite connectivity (see Example $2$ in \cite{Dominguez 1998}). Later in 2008, Rippon and Stallard constructed an example of a meromorphic function with finitely many poles with a bounded doubly connected wandering domain such that each iterated image of the wandering domain is bounded and simply connected (see Example $2$ in \cite{RipponStallard 2008}). Moreover, they proved that for suitably small values of $a$ and $\epsilon$, the function $2+2z-2e^z+\frac{\epsilon}{e^z-e^a}$ has a wandering domain such that each iterated image of the wandering domain is bounded and infinitely connected but the wandering domain is not itself a BWD (see Example 1, \cite{RipponStallard 2008} for details).   
		
		\par  Rippon and Stallard provide a necessary and sufficient condition for a meromorphic function with finitely many poles to have BWDs.  Recall that, for a Fatou component $U$ of $f$, $U_n$ denotes the Fatou component containing $f^n (U)$.
		
		\begin{theorem}(Theorem 1, \cite{RipponStallard 2008})
			Let $f$ be a meromorphic function with finitely many poles and $U$ be a multiply connected wandering domain of $f$. Then
			\begin{enumerate}
				\item $U$ is a BWD if and only if $U_n$ is multiply connected for infinitely many values of $n$. 
				\item If $~\mbox{sing}~(f^{-1})\cap \cup_{n\geq 1} U_n= \emptyset$ then $U_n$ is multiply connected for all $n$ and therefore $U$ is a BWD.
			\end{enumerate}
			\label{MC-IFF-BWD-FP}	\end{theorem} 
		
		The next example stands apart from the previous ones and also is the first example of a BWD in the presence of infinitely many poles \cite{RipponStallard 2005(1)}. We utilize the following fact (see Lemma 7 in \cite{Berg 1993} and Equation (1.1), \cite{RipponStallard 2005(1)}) in the next example.
		\begin{lemma}
			Let $U$ be a periodic Fatou component $U$ of a meromorphic function $f$ such that $\lim_{n \to \infty}f^n (z) =\infty$ for each $z \in U$ then there exists $M>0$ such that   $\frac{\log \log |f^n(z)|}{n} \leq M$   for all sufficiently large $n$.
			\label{growth-on-periodicFC}
		\end{lemma}  
		
		\begin{example}
			There is a meromorphic function with infinitely many poles which has a BWD.
			\label{BWD-infinitelymanypoles}
		\end{example}
		\begin{proof}

			Consider $F_\infty (z)=\left(1+\sum_{n=1}^{\infty}\frac{1}{z-r_n}\right) f(z)$ where $f(z)$ is the entire function discussed in Example \ref{1st exm}. The point $-r_n$ is a zero and $r_n$ is a pole for each $n$. We first show that  $\left(1+\sum_{n=1}^{\infty}\frac{1}{z-r_n}\right)\to1$ as $z\to\infty$ through the annuli $A_n=\{z:r_n^2<|z|<\sqrt{r_{n+1}}$. Note that there is neither a zero nor a pole in this annulus. For this, note that
			\begin{equation}\label{Eq13}
				\sum_{k=1}^{\infty}\frac{1}{z-r_k}=\sum_{k=1}^{n}\frac{1}{z-r_k}+\frac{1}{z-r_{n+1}}+\sum_{k=n+2}^{\infty}\frac{1}{z-r_k}.
			\end{equation} For  $k \leq n,  $ we have $r_k \le r_n$. Since $|z| > r_n^2$, 
			$|z - r_k| \ge |z| - r_k \ge r_n^2 - r_n \ge \tfrac{r_n^2}{2}$ for all sufficiently large $n$, we get
			$$\left| \sum_{k=1}^{n} \frac{1}{z - r_k} \right|
			\le \sum_{k=1}^{n} \frac{1}{|z - r_k|}
			\le\frac{2n}{r_n^2} < \frac{2n}{2^{2(n-1)}r_1 ^2}.$$ 
			The last inequality follows from Inequality~(\ref{growth-rn}).
			
			We also have, 	$|z - r_{n+1}| \ge r_{n+1} - |z| \ge r_{n+1} - \sqrt{r_{n+1}} \ge \tfrac{1}{2}r_{n+1}
			$ for all sufficiently large $n$. Similarly, for $k \geq n+2$,  we have $r_k \ge r_{n+2}$. Since  $|z| < \sqrt{r_{n+1}}$,
			$|z - r_k| \ge r_k - |z| \ge r_{n+2} - \sqrt{r_{n+1}} \ge \tfrac{C r_{n+1}^2}{2}$ for all sufficiently large $n$. Consequently,
			$$\left|\sum_{k = 1}^{\infty}  \frac{1}{z - r_k} \right|	
			\le \frac{2n}{2^{2(n-1)}r_1 ^2}+ \frac{2}{r_{n+1}} + \frac{2}{Cr_{n+1}^2}.$$
			
			As $z\to\infty$   through   $A_n $, $r_n \to \infty$. Therefore $\sum_{k=1}^{\infty}\frac{1}{z-r_k}\to 0$ and  $\left(1+\sum_{n=1}^{\infty}\frac{1}{z-r_n}\right)\to1$ as $z\to\infty$   through $A_n $.
			\par 
			Following the arguments used in Example~\ref{Exm 3}, we conclude that $F_\infty(A_n) \subset A_{n+1}$ for all sufficiently large $n$. This implies that there are Fatou components, say $U_n$, containing these $A_n$s.
			
			\par For $z \in A_n$, let $a_k = \log |F_{\infty}^{k}(z)|$. Then $a_{k+2}>\log  r_{n+k+2}^2 $ and $a_k < \log  \sqrt{r_{k+n+1}}$. Consequently, $ \frac{a_{k+2}}{a_k}> \frac{4 \log r_{n+k+2} }{\log r_{k+n+1}}.$
			It follows from Equation~(\ref{logrn-limit}) that $\lim_{k \to \infty}\frac{a_{k+2}}{a_k}=\infty$.

			Let $M>1$. Then there exists a $k_0$ such that $a_{k+2}> Ma_{k}$ for all $k \geq k_0$. This gives that $a_{2k}> M^{k-k_0}a_{2k_0}$. Taking the logarithm and then dividing by $k$ on both the sides, we have
			$\frac{\log a_{2k}}{k} >\frac{\log (\frac{a_{2k_0}}{M^{k_0}})}{k}+ \log M$. This means that  $\lim_{k \to \infty} \frac{\log \log {|F_{\infty}^{2k}(z)|}}{k} =\infty,$  for each $z \in A_n$ and for all sufficiently large $n$. In view of Lemma~\ref{growth-on-periodicFC}, $U_n$, the Fatou component containing $A_n$, is not periodic for any $n$. In other words, $U_n$ is wandering and hence a BWD.
		\end{proof}
		In view of Examples \ref{Exm 3} and \ref{BWD-infinitelymanypoles}, the following question arises.
		\begin{Question}
			If $f$ is an entire function with BWD $U$ and $g$ is a meromorphic function such that $\lim_{z \to \infty}g(z)=1$ along $U_n$, then is it always true that the product $fg$ has BWD?
		\end{Question}
		
		A related question is also raised by Zheng in 2010, probably in a conference in Warsaw.
		\begin{Question}
			Let an entire function $f$ have BWD and $Q$ is a rational function such that $\lim_{z \to \infty} Q(z)$ exists and is finite. Is it always true that $f+Q$ has a BWD? 
		\end{Question}

		\section{Singular values and Dynamics}
		This section discusses the singular values and dynamics of a function in the presence of a BWD.
		
		\subsection{Singular values}\label{singularvalues-section}
		A point $a\in\widehat{\mathbb C}$ is called a singular value of a meromorphic function $f$ if for every open neighborhood $U$ of $a$, there exists a component $V$ of $f^{-1}(U)$ such that $f: V\to U$ is not injective, meaning that at least one branch of $f^{-1}$ fails to be defined at the point $a$. This is why a singular value of $f$ is also called as a singularity of $f^{-1}$. There are different possible ways in which this failure can take place, leading to the following classification  \cite{BergEreme 1995}. 
		
		For $a\in \mathbb{{\widehat{C}}}$ and $r > 0$, let $D_r(a)$ be a disk (in the spherical metric) and choose a component $U_r$ of $f^{-1}(D_{r}(a))$ in such a way that $U_{r_{1}}\subset U_{r_{2}}$ for $0 < r_1 < r_2$. There are two possibilities. 
		\begin{enumerate}
			\item If $\bigcap_{r>0}U_{r}=\{z\}$ for $z\in \mathbb{C},$ then $f(z)=a$.  If $z$ is a multiple pole or  $f'(z)= 0$ and $a\in \mathbb C$, then $z$ is a critical point.
			In this case, $a$
			is called a critical value and we say that a critical point.
			\item If $\bigcap_{r>0}U_{r}=\emptyset$ then we say that the choice $r\mapsto U_r$ defines a transcendental singularity of $f^{-1}$. We say that a transcendental singularity lies over $a$ and $a$ is an asymptotic value of $f$. In this case, there is an unbounded curve $\gamma: (0, \infty) \to \mathbb{C}$ such that $\gamma(t) \to \infty$ as $t \to \infty$ and $\lim_{t \to \infty}f(\gamma(t)) =a$.
			The singularity lying over $a$ is called direct if there exists $r>0$ such that $f(z)\neq a$ for all $z\in U_r$. A singularity is indirect if it
			is not direct. The singularity lying over $a$ is called logarithmic if $f: U_r\to {D_r(a)\setminus\{a\}}$ is a universal covering for some $r>0$.
		\end{enumerate} 
		The relationship between the singular values and periodic Fatou components of a meromorphic function is well-known (see Theorem 7 in \cite{Berg 1993} for further details).  The post-singular set of a function $f$ is the union of forward orbits of all its singular values as long as these are defined. Bergweiler et al., in 1993, proved that if  $U$ is a wandering domain of an entire function $f$, then all the limit functions of $\{f^n\}_{n>0}$ on $U$ are contained in the union of the derived set of the post-singular set and $\{\infty\}$ \cite{Bergetal 1993}. Later, in 2002, Zheng extended this for meromorphic functions with finitely many poles (Theorem 4, \cite{Zheng 2002}).  
		\par We need the notion of exponent of convergence in order to state a relation between BWD and singular values of a function.  The exponent of convergence of zeros of a transcendental entire function $f$,  denoted by $\lambda(f)$ is defined as \begin{equation}
			\inf\left\{\lambda>0:\sum_{n=1}^{\infty}\frac{1}{|a_n|^\lambda}<\infty, f(a_n)=0,  a_n \neq 0 \right \}.
			\label{exponentofconvergence}\end{equation}  
		
		The exponent of convergence of zeros of a function does not make sense if the function has only finitely many zeros. It is known that  $\lambda(f) \leq \rho(f)$ where $\rho(f)$ denotes the order of $f$ (Theorem 2.5.18, \cite{Boas-1954}).  In 2002, Cao and Wang proved the following when  $\lambda(f) $ is strictly less than  $ \rho(f)$.
		\begin{theorem}(Theorem 2, \cite{CaoWang 2002})
			Let $f$ be an entire or meromorphic function with finitely many poles and $\lambda(f)<\rho(f)$. If $U$ is a BWD of $f$, then $U_j\cap sing(f^{-1})\neq\emptyset$ for some $j$.
			\label{sing-BWD}\end{theorem}
		The above theorem guarantees the presence of a singular value  in a BWD, which is crucial in determining its eventual connectivity. One may ask the following.
		\begin{Question}
			Does Theorem~\ref{sing-BWD} hold for meromorphic functions with infinitely many poles?
		\end{Question}
		A  result by Bergweiler et al. published in 2013 completely determines the connectivity of a BWD $U$ in terms of the number of critical points contained in its forward orbit $\cup_{n>1} U_n$. Recall that $c(V)$ denotes the connectivity of the Fatou component $V$. 
		\begin{theorem}(Theorem 1.7, \cite{BergRippStall 2013})
			Let $f$ be an entire function having a BWD $U$. Then 
			\begin{enumerate}
				\item $c(U)=2$ if and only if $\cup_{n=1}^{\infty}U_n$ contains no critical point of $f$;
				\item $2<c(U)<\infty$ if and only if the number of critical points of $f$ in $\cup_{n=1}^{\infty}U_n$ is non-zero and finite;
				\item $U$ is infinitely connected if and only if $\cup_{n=1}^{\infty}U_n$ contains infinitely many critical points of $f$.
			\end{enumerate}
			\label{connectivity-criticalpoints}
		\end{theorem}
		
		The main ingredient of the proof of Theorem~\ref{connectivity-criticalpoints} is the Riemann-Hurwitz formula. Though it seems reasonable to believe that this theorem is true in the presence of poles, the immediate next  question should be the following.
		\begin{Question}
			Is Theorem~\ref{connectivity-criticalpoints} true for meromorphic functions with finitely many poles?
		\end{Question}
		There is a corollary to Theorem~\ref{connectivity-criticalpoints}.
		\begin{corollary} The BWDs discussed in Example~\ref{2nd exm} are infinitely connected.
			\label{SecondExample-infinitelyconnected}
		\end{corollary}
		A nice and interesting situation would be to have all the critical points in the forward orbit of a BWD. Nice because all the critical points are escaping, and interesting because the function will be of unbounded type.  
		\begin{Question}
			Does there exist an entire function with an infinitely connected BWD $U$ such that each critical point of $f$ is contained in $U_k$ for some $k$?
		\end{Question}
		A point $a\in\mathbb{\widehat{C}}$ is said to be an omitted value of a meromorphic function   $f: \mathbb{C} \to \widehat{\mathbb{C}}$ if $f(z)\neq a$ for any $z\in\mathbb{C}$. In 1914, Iversen \cite{Iversen 1914} proved that every omitted value is an asymptotic value. For every neighborhood $D$ of an omitted value $a$,  every component
		$C$ of $f^{-1}(D)$  is unbounded (see Lemma 2.1, \cite{Nayak 2015}). Every singularity lying over an omitted value is transcendental.
		\par Since $\infty$ is an omitted of every (transcendental) entire function $f$,  the pre-image $f^{-1}(D)$ of each neighborhood $D$ of $\infty$ is unbounded. But the set $f^{-1}(D)$ is not connected in general. However, if $f$ has a BWD then $f^{-1}(D)$ is not only connected and unbounded but also satisfies a significant property, namely $f^{-1}(D)$ is infinitely connected in such a way that each of its boundary components is bounded. This was first observed by  Chakra et al., \cite{ChakraChakrabortyNayak 2016} in 2016, who introduced the notion of a Baker omitted value. This also makes sense for meromorphic functions with poles.
		\begin{definition}[Baker omitted value]
			An omitted value $a\in\mathbb{\widehat{C}}$ of a meromorphic $f$ is said to be a Baker omitted value (in short, bov) if there is a disk $D$ with centre at $a$ such that each component of the boundary of $f^{-1}(D)$ is bounded. 
		\end{definition} 
		If each boundary component of $f^{-1}(D)$ is bounded in the above definition then it can in deed be shown that $f^{-1}(D)$ is connected with infinitely many complementary components. Further, this is true for every disk (in fact, every simply connected domain) around a bov (see Lemma 2.3, \cite{ChakraChakrabortyNayak 2016}).
		
		It is known that if a function $f$ has a bov then it is the only asymptotic value of $f$ (see Theorem 2.1 in \cite{ChakraChakrabortyNayak 2016}). If $f$ is an entire function, then $\infty$ is an asymptotic value and is the only candidate for bov. However, a pole makes the situation very different.
		\begin{Remark}
			If a meromorphic function with at least one pole has bov, then the bov is a finite complex number. This is because $\infty$ is not omitted in this case.
			\label{BOV-mero}
		\end{Remark}  The following result of Chakra et al. underlines the connection between a bov and BWD. 
		\begin{theorem}(Theorem 2.3, \cite{ChakraChakrabortyNayak 2016})
			If an entire function has a BWD, then it has a Baker omitted value, and that is $\infty$.
			\label{BOV-BWD}
		\end{theorem}
		The converse of Theorem~\ref{BOV-BWD} is not true in general.
		Several examples of entire functions with a Baker omitted value are known for which there is no BWD. Examples of such functions include $e^z +z +\lambda, \lambda \geq 0$ \cite{ChakraChakrabortyNayak 2016}. Later, it was seen that $e^z +P(z)$ is such a function for every non-constant polynomial \cite{DasNayak 2024}. This leads to the following question.
		\begin{Question}
			Find a sufficient condition for an entire function with Baker omitted value to have a BWD.
		\end{Question}

		Though Theorem~\ref{BOV-BWD} does not hold for meromorphic functions, something in the same spirit remains true. To state this precisely, we make a definition.
		\begin{definition}[Local Baker omitted value at $\infty$]
			A meromorphic function $f$ has a local bov at $\infty$  if, for a disk $D$ around $\infty$, $f^{-1}(D)$ has exactly one unbounded component $C$  and each component of its boundary is bounded.
		\end{definition}
		Here $f^{-1}(D)$ can be disconnected unlike in the case of bov, but every other component of $f^{-1}(D)$ (necessarily bounded) contains at least one pole.  
		An entire function has a local bov; then  $\infty$ is actually the bov.
		Here is a remark.
		\begin{Remark}
			For a meromorphic function $f$, there is a non-logarithmic singularity of $f^{-1}$ over $\infty$ whenever $\infty$ is a local bov.
		\end{Remark}
		
		A simple but useful observation follows.
		\begin{prop}
			If $f$ is meromorphic with BWD, then the image of every unbounded curve under $f$ is unbounded.  In particular, there is no BWD if the function has a finite asymptotic value.
			\label{unbounded-BWD}
		\end{prop}
		The proof of Proposition~\ref{unbounded-BWD} is already available for entire functions (see proof of Theorem 2.3, \cite{ChakraChakrabortyNayak 2016}), which  works even when there is a pole.
		\begin{proof}[Proof of Proposition \ref{unbounded-BWD}]
			Let $U$ be a BWD of $f$. It follows from the definition of BWD that,  every unbounded curve $\gamma$  intersects $U_n$ for all large $n $. Since  $U_{n+1}$ surrounds $U_n$, the set $f(\gamma)$ intersects  $U_n$ also for all large $n$. This gives that $f(\gamma)$ is unbounded.
		\end{proof}
		There is an interesting though trivial consequence of Proposition~\ref{unbounded-BWD} and Remark~\ref{BOV-mero}.
		\begin{prop}
			If a meromorphic function with a pole has bov, then it has no BWD. 
		\end{prop}
		We now state and prove the generalization of Theorem~\ref{BOV-BWD}.
		\begin{theorem}
			If $f$ is meromorphic with at most finitely many poles and has a BWD, then it has a local bov at $\infty$.
			\label{BWD-LOCALBOV}
		\end{theorem}
		\begin{proof} Let $D$ be a disk around $\infty$. Then $f^{-1}(D)$ is clearly unbounded (by Picard's theorem). 
			\par
			
			Since each bounded component of $f^{-1}(D)$ has to contain a pole and the number of poles is finite, there is at least one unbounded component of $f^{-1}(D)$. 
			
			\par	
			If the number of unbounded components of $f^{-1}(D)$ is at least two, then one of them has an unbounded boundary component. Let this boundary component be denoted by $\gamma$. This $\gamma$ is mapped into the boundary of $D$ by $f$. The boundary of $D$ is a bounded set. On the other hand, its image $f(\gamma)$ is unbounded by Proposition~\ref{unbounded-BWD}. This is a contradiction.
			Therefore, the set $f^{-1}(D)$ has a unique unbounded component, say $C$. Now, using Proposition~\ref{unbounded-BWD} again, we conclude that every component of the boundary of $C$ is bounded.
		\end{proof} 
		
		The following question stems from Theorem~\ref{BWD-LOCALBOV}.
		\begin{Question}
			Is it always true that for a meromorphic function with infinitely many poles and BWD, the point $\infty$ is a local bov?
		\end{Question}
		An analysis of Example~\ref{BWD-infinitelymanypoles} may be useful to answer the above question. 
		\subsection{Julia components}
		As discussed in the introduction, in spite of being of unbounded type, meromorphic functions with BWD are with inherent advantages as long as their dynamics is concerned. One such is the following.	
		\begin{prop}
			If $f$ is a meromorphic function with BWD, then every point in the backward orbit of $\infty$, i.e., $\{z \in \widehat{\mathbb{C}}: f^m (z)=\infty~\mbox{for some}~m \geq 1\}$, is a singleton and buried component of the Julia set.
			\label{pre-poles}
		\end{prop} 
		\begin{proof}
			It follows from the definition of BWD that the Julia component containing $\infty$ is a singleton and buried. Let $z$ be such that $f^m (z)=\infty$ for some $m\geq 1$ and  $J$ be the Julia component containing $z$. Then $f^k(J)$ contains $\infty$ for some $k \leq m$. Choose the smallest such $k$. Consider a sequence of Jordan curves $\gamma_n$, each belonging to a distinct BWD such that each $\gamma_n$ surrounds the origin and $\cap_{n \geq 1} \mathbb{C} \setminus B(\gamma_n) = \emptyset$, where $B(\gamma_n)$ is the set of all points surrounded by $\gamma_n$. Then there is a sequence of closed curves $\alpha_n$ each surrounding $z$ such that $f^k (\alpha_n) \subseteq \gamma_n$.  Since $f$ is meromorphic, $\cap_{n \geq 1} B(\alpha_n) =\{z\}$. As each $\alpha_n$ is in the Fatou set - in fact, in a BWD, the set $J$ must be singleton and buried.
		\end{proof}
		
		The argument used in the proof of Proposition~\ref{pre-poles} in fact gives that the set of all singleton and buried Julia components is backward invariant, i.e., if $J$ is a singleton buried Julia component and $J_{-1}$ is a Julia component such that $f(J_{-1}) \subseteq J$ then $J_{-1}$ is singleton and buried. It can also be seen that the set of all singleton and buried Julia components is completely invariant. The complement of this set in the Julia set is also completely invariant. Existence of singleton buried Julia components for entire functions, where the backward orbit of $\infty$ is empty, was reported by Dominguez in 1997.
		\begin{theorem}[Theorem 8.1, \cite{Dominguez 1997}]
			If an entire function has BWD, then singleton and buried components are dense in its Julia set.  
		\end{theorem}
		There are Julia components $J$ (such as those containing the repelling periodic points of the function) with bounded forward orbits, i.e., $\cup_{n>0} f^n(J)$ is a bounded set.  Clearly, the Julia component containing the boundary of a BWD is not a singleton and is contained in the escaping set of the function.
		There are recent results by Kisaka relating the topology of Julia components of entire functions with BWD and their behaviour under iteration. A Julia component is called full if its complement in $\mathbb{C}$ is connected. 
		\begin{theorem}[Theorem A, \cite{Kisaka2025}]
			Let $f$be  an entire function with BWD. Then for every Julia component $C$ of $f$ with bounded forward orbit, the following are true.
			\begin{enumerate}
				\item  $C$ is quasiconformally homeomorphic to a Julia component of a polynomial.
				\item If $C$ is full, then it is buried.
				\item If $C$ is not full then each bounded component of its complement consists of either an attracting domain, a parabolic domain, a Siegel disk or one of their pre-images.
				\item If $C$ is wandering then it is singleton as well as buried.
			\end{enumerate}
			
			\label{Kisaka2025} 		
		\end{theorem}
		The main idea of the proof of Theorem~\ref{Kisaka2025} is that $f^k$ for some $k \geq 1$ is a polynomial-like map on suitable domains (see Main Lemma in \cite{Kisaka2025}). However, this observation is already made by Zheng (see Proof of Theorem 3, \cite{Zheng 2000}). Indeed, Theorem E of \cite{Kisaka2025} is a restatement of Theorem 3 of \cite{Zheng 2000}.   A natural question arises in the backdrop of Theorem~\ref{Kisaka2025}.
		\begin{Question}
			Let $C$ be a  Julia component of an entire function $f$ with BWD. Also, let $C$ be disjoint from the boundary of every BWD and its forward orbit is unbounded. Then, are the following true?
			
			\begin{enumerate}
				\item If $C$ is full, then it is buried.
				\item If $C$ is wandering then it is singleton as well as buried. 
			\end{enumerate} 
		\end{Question}
		There are two possibilities for $C$ in the above question, namely that all the points of $C$  are escaping, or none are escaping. This observation may be useful for answering the question. 
		\subsection{Escaping set}
		The escaping set for an entire function $f$ is defined as $I(f)=\{z: \lim_{n \to \infty}f^n (z)=\infty\}$. 
		In 1989, Eremenko proved that $\mathcal{J}(f)=\partial I(f)$, $\mathcal{J}(f)\cap I(f)\neq\emptyset$ and $\overline{I(f)}$ has no bounded components  \cite{Eremenko 1989}.  In 1989, Eremenko conjectured that $I(f)$ has no bounded component. Later in 1999, Bergweiler and Hinkkanen \cite{BergHinkk 1999} introduced an important subset $A(f)$ of the escaping set, called the fast escaping set, defined by
		$$A(f)=\{z \in I(f):~\mbox{there exists}~L\in\mathbb{N}~\mbox{such that }  |f^n(z)|>M(R, f^{n-L}), ~\mbox{for all}~n>L \}$$
		where $M(r,f)=\max_{|z|=r} |f(z)|$ and $R$ is any value such that $R>\min_{z\in\mathcal{J}(f)}|z|$. The set $A(f)$ is non-empty and completely invariant. Therefore, $\mathcal{J}(f)\subset\overline{A(f)}$. Further, if $f$ does not have any wandering domain, then $A(f)\subset\mathcal{J}(f)$. All these results can be found in ~\cite{BergHinkk 1999}. In 2005, Rippon and Stallard proved that every component of $A(f)$ is unbounded (see Theorem 1, \cite{RipponStallard 2005}). This gives that $I(f)$ has at least one unbounded component providing some evidence supporting Eremenko's conjecture.  The following theorem by Rippon and Stallard shows that the conjecture is true for entire functions in the presence of BWD.
		
		\begin{theorem}(Theorem 2, \cite{RipponStallard 2005})
			If $f$ is an entire function having a BWD, then
			\begin{enumerate}
				\item the sets $A(f)$ and $I(f)$ are connected and unbounded;
				\item the closure of each BWD  is contained in  $A(f)$.
			\end{enumerate}
		\end{theorem}
		
		The first conclusion of the above theorem gives that the escaping set of an entire function with BWD contains all BWDs along with some connected sets, each of which intersects the Julia set and joins two BWDs. 
		\par For a BWD $U$ of an entire function $f$, the set $\cup_{n\geq 1}f^n(U)$ is clearly unbounded. This leads to an interesting question.
		\begin{Question}
			For a simply connected wandering domain  $W$ of an entire function $f$, is it always true that $\cup_{n\geq 1}f^n(W)$ is unbounded?
			\label{WD-orbunbded}
		\end{Question} 
		This question can be reworded as: \textit{Is every simply connected wandering domain of an entire function is either escaping or oscillating?} In 1984, Eremenko and Lyubich \cite{EremenkoLyubich 1984} asked this question in general whether transcendental entire functions have wandering domain whose forward orbits is bounded. In 2000, Zheng answered this question positively under the hypothesis that the function has a BWD.
		\begin{theorem}(Theorem 3, \cite{Zheng 2000})
			Let $f$ be an entire function having a BWD. Then for every wandering domain $U$ of $f$, there is a subsequence $ {n_k}$ of natural numbers such that $f^{n_k}(z)\to\infty$ uniformly as $k \to \infty$ on $U$. In particular, $\bigcup_{n \geq 1}U_n$ is unbounded.
			\label{scwd-bded}
		\end{theorem} 
		In 2024, Pardo-Sim\'on and Sixsmith constructed a simply connected wandering domain with the property that nearly all of its forward iterates are contained within a bounded domain, in \textit{some sense}. The precise results can be found in \cite{Simon Sixsmith 2024}. The sanme question was posed by Bergweiler in 1993 for meromorphic functions, possibly with poles (see Question 8 in \cite{Berg 1993}).  In 2026, Prochorov and his coauthors settled this longstanding question for transcendental entire functions, proving that they do not admit bounded-orbit wandering domains \cite{PRW 2026}. Parallel and Independent work by Drach and his coauthors also established the same result, extending it to both entire and meromorphic functions \cite{DSU 2026}.

		\par
		An important consequence of BWDs of a meromorphic function   with finitely many poles is due to Zheng, who proved the following in 2006.
		\begin{theorem}(Theorem, page-25,  \cite{Zheng 2006}\label{Thm15})
			Let $f$ be a meromorphic function with finitely many poles and have a BWD $U$.  If $B\subset U$ is a domain containing a closed curve which is not null-homotopic in $U$, then for all sufficiently large $n$, $A(r_n, R_n)\subset f^n(B)\subset U_n$, where $r_n, R_n >0$ for each $n$ such that  $\lim_{n \to \infty}\frac{R_n}{r_n} = \infty$. 
		\end{theorem}
		Theorem~\ref{Thm15} is true for entire functions as mentioned in \cite{BergRippStall 2013}. The set $B$ in this theorem is multiply connected.  
		Later in 2013, what Bergweiler et al. \cite{BergRippStall 2013} proved for entire functions gives that the iterated images of any domain (without any restriction on its connectivity)  contained in a BWD must contain annuli with increasing modulus, which is not necessarily the case for Theorem~\ref{Thm15}. 
		\begin{theorem}(Theorem 1.2, \cite{BergRippStall 2013}\label{increasing-annulus})
			Let an entire function $f$ have a BWD $U$. Then, for each $z_0\in U$ and each open set $B\subset U$ containing $z_0$, there exists $0<\alpha<1$ such that, for all sufficiently large $n$, 
			$A(r_n, R_n)\subset f^n(B)\subset U_n $ 
			where $r_n=|f^n(z_0)|^{1-\alpha}$ and $R_n=|f^n(z_0)|^{1+\alpha}$. Further,  $\liminf_{n\to\infty}\frac{\log R_n}{\log r_n}>1$. 
		\end{theorem}
		
		From Theorem \ref{increasing-annulus}, it follows that if $U$ is a BWD and $z_0\in U$ then there exists $0<\alpha<1$ such that, for all sufficiently large $n $, the maximal annulus centered at $0$ that is contained in $U_n$ and contains $f^n(z_0)$ is of the form $B_n=A(r^{\alpha_n}_n, r^{\beta_n}_n)$ where $r_n=|f^n(z_0)|$ and $0<\alpha_n<1-\alpha<1+\alpha<\beta_n$ for some sequence of positive reals $\{\alpha_n\}_{n\geq 1}$ and $\{\beta_n\}_{n\geq 1}$. The annuli $B_n$ are the sets in which the iterates of all the points of $U$ eventually lie, and the union of these $B_n$s acts as an `absorbing set' for $f$. More precisely,  Bergweiler et. al. proved the following.
		\begin{theorem}(Theorem 1.3, \cite{BergRippStall 2013}\label{Thm 2.14})
			Let $U$ be a BWD of an entire function $f$ and $z_0\in U$. Then for each compact subset $C$ of $U$, $f^n(C)\subset B_n$ for all large $n$, where $B_n$ is defined as above. 
		\end{theorem}
		\par Bergweiler et al. provided a necessary and sufficient condition for the existence of BWD for meromorphic functions having a direct singularity over $\infty$. Such functions can have infinitely many
		poles.
		\begin{theorem}{(Theorem 1.3, \cite{BergRippStall 2008})}
			Let f be a meromorphic function with a direct singularity over $\infty$. Then $f$ has BWD if and only if all the components of $I(f)\cap J(f)$ are bounded.
		\end{theorem}	
		
		\subsection{Eventual connectivity}
		\par
		We begin with a word on Fatou components. If $U$ is a Fatou component of a meromorphic function $f$ with BWD then the Fatou component $U_n$ containing $f^n (U)$ is exactly $f^n(U)$. This is because every point of $ U_n \setminus f^n(U)$ is a finite asymptotic value of $f$ (see Theorems 1 and 2, \cite{Herring-1998}). But no finite asymptotic value can exist in the presence of a BWD (Proposition~\ref{unbounded-BWD}).  It is important to note that $U$ is not necessarily a BWD. In fact, for a meromorphic function $f$ with BWD, if $U$ is any Fatou component of $f$, then $f: U \to U_1$ is proper.
		This is one of the reasons why the tools developed to study dynamics of rational maps can be expected to be used for investigating the dynamics of functions with BWD. One such tool is the Riemann-Hurwitz formula. For a BWD $U$, if $d_n$ is the degree of the proper map $f^n: U\to U_n$ and the  connectivities of $U$ and $U_n$  are finite, then by the Riemann-Hurwitz formula we have for each $n$, 
		$$c(U)-2=d_n\left(c(U_n)-2\right)+c_f$$ where $c_f$ denotes the number of critical points of $f^n$ in $U_n$ counting multiplicity, and   $c(U)$ and $c(U_n)$ denote the connectivities of $U$ and $U_n$ respectively. It follows that $c(U_n)\leq c(U)$ for all $n $, which gives rise to a natural question: How does the connectivity of a BWD evolve under the iteration of $f$, i.e., what happens to the sequence $c(U_n)$ as $n \to \infty$?  
		\begin{definition}[Eventual connectivity]
			A natural number $c$ is called the eventual connectivity of a BWD $U$ if  $c(U_n)=c$ for all sufficiently large $n$. 
		\end{definition}
		
		We first discuss Baker's example briefly. The question of whether the connectivity of the BWD appearing in Example~\ref{1st exm} is finite or not was raised by
		Baker himself in \cite{Baker 1988} and by Kisaka and Shishikura in \cite{KisakaShishikura 2008}. Later in 2011, Bergweiler and Zheng proved that the connectivity of the BWD in the above example is infinite by using Lemma 2.2 and Theorem 1.2 of \cite{BergZheng 2011} (see Section 6, \cite{BergZheng 2011} for a similar example). 
		
		\begin{theorem}(Theorem A, \cite{KisakaShishikura 2008})\label{connectivity}
			Let an entire function $f$ have a BWD $U$. 
			\begin{enumerate}
				\item If $c(U_k)=\infty$ for some $k$ then the eventual connectivity of $ U $ is $\infty$.
				
				\item If $c(U_k)<\infty$ for some $k$  then the eventual connectivity of $U_k$ is $2$. Moreover, if 
				$c(U_k)=2$ for some $k$  then $\cup_{n=k}^{\infty}U_n$ does not contain any critical point of $f$.
			\end{enumerate}
		\end{theorem}
		It follows from a result of Bolsch (Theorem $3$, \cite{Bolsch 1999}) that if $f: U \to V$ is a proper analytic map for two domains $U, V$, then $V$ is infinitely connected if and only if $U$ is infinitely connected. This implies that the eventual connectivity of a BWD $U$ is $\infty$ if and only if all the BWDs in the grand orbit of $U$ are infinitely connected. However, for each $m$, there are BWDs with connectivity $m$ such that their eventual connectivity is $2$ (Theorem C, \cite{KisakaShishikura 2008}). In the same paper, Kisaka and Shishikura   gave the first example of an entire function having a BWD whose eventual connectivity is $2$ (see Theorem $B$). Using Theorem~\ref{connectivity},  the eventual connectivity of the BWD discussed in Example \ref{2nd exm} is found to be infinity.

		In 2008, Rippon and Stallard generalized the work of Kisaka and Shishikura for meromorphic functions with finitely many poles.
		\begin{theorem}(Theorem 3, \cite{RipponStallard 2008})
			Let $f$ be a meromorphic function with finitely many poles and $U$ be a wandering domain of $f$.
			\begin{enumerate}
				\item If $U$ is a BWD then the eventual connectivity of $U$ is either $2$ or $\infty$.
				\item If $U$ is not a BWD, then the eventual connectivity of $U$ is $1$.
			\end{enumerate}
			\label{eventual-mero}
		\end{theorem}
		To the best of our knowledge, Example~\ref{BWD-infinitelymanypoles} is the only example of a meromorphic function with infinitely many poles that has a BWD whenever $k >2$. Here is a question on this example.
		
		\begin{Question}
			What is the eventual connectivity of the BWD discussed in Example~\ref{BWD-infinitelymanypoles}? 
			\label{EC-Infinitely-many-poles}
		\end{Question}  
		To answer Question~\ref{EC-Infinitely-many-poles}, one has to locate the critical points.   
		\par 
		It is pointed out in \cite{RipponStallard 2008} (page $408$) that, for any given  $k \geq 2$,  a meromorphic function can be constructed having a BWD with eventual connectivity $k$. These functions must be having infinitely many poles in view of Theorem~\ref{eventual-mero}.
		\par Recall that the forward orbit of every point belonging to a BWD tends to infinity. Recently, the behavior of forward orbits of two points belonging to a single BWD has also been studied in \cite{gustavo-2022, gustavo-2024}. The hyperbolic metric has been a main tool in these investigations.
		\section{Functions without any Baker wandering domain}
		This section discusses several conditions  ensuring non-existence of BWDs. The first such condition was provided by Baker himself in 1984. He proved that if an entire function  $f$ is bounded on an unbounded curve, then $f$ has no BWD (Corollary, page 565, Section 3, \cite{Baker 1984}). That this is also true for meromorphic functions follows from Proposition~\ref{unbounded-BWD}.
		Later in 1993, Bergweiler generalized this result for entire functions by weakening the hypothesis, as follows. 
		\begin{theorem}(Theorem 10, \cite{Berg 1993})\label{Thm 4.2}
			Let $f$ be an entire function, and for each $\epsilon>0$ there exists an unbounded curve $\gamma$ such that $|f(z)|\leq M(|z|^{\epsilon}, f)$ for $z\in\gamma$; then all the Fatou components of $f$ are simply connected. In particular, there is no BWD for $f$.
		\end{theorem}
		The function $f(z)=z+e^z$ satisfies the hypothesis of Theorem \ref{Thm 4.2}. To see this, let $\epsilon>0$. Then  choose $\gamma=\{t \in \mathbb{R}: -\infty < t < -M_\epsilon\}$ for suitable $M_\epsilon >1$ so that $|z|< e^{|z|^\epsilon}$ for every $z \in \gamma$. This is possible as $\lim_{|z| \to \infty} \frac{\log \log |z|}{\log |z|} =0.$ Further, for all $z \in \gamma, $ we have $z< e^z +z <0$, which gives  $|e^z +z|<|z|$. Therefore, $|f(z)|< |z|< e^{|z|^\epsilon}<e^{|z|^\epsilon} +|z|^\epsilon$ and the upper bound is nothing but $ M(|z|^{\epsilon}, f)$. 
		\par There is an extension of Theorem~\ref{Thm 4.2} by Zheng. 
		\begin{theorem}[Corollary 2 (I), \cite{Zheng 2006}]
			If $f$ is a meromorphic function with at most finitely many poles such that for every $\epsilon>0$, there is an unbounded curve $\gamma$  such that $\log |f(z)| < \epsilon \log M(|z|,f)$ for all $z \in \gamma$, then $f$ does not have any BWD.
		\end{theorem} 
		In the same paper, Zheng also proved the following result.
		\begin{theorem}{(Corollary 5, \cite{Zheng 2006})}
			If $f$ is a meromorphic function with finitely many poles such that for all sufficiently large $r>0$ and $d>1$, $\log M(2r,f)>d\log M(r,f)$, then
			$f$ does not have any BWD. 
		\end{theorem} 
		\par Using the logarithmic change of variable, in 1992, Eremenko and Lyubich proved that there is no Fatou component for any entire function $f$ in class $\mathcal{B}$ such that the iterates of $f$ tend to infinity \cite{EremenkoLyubich 1992}. In 2000, Zheng generalized this result for meromorphic functions.
		\begin{theorem}(Theorem 2, \cite{Zheng 2000})
			Let  $f $ be a meromorphic function of bounded type, i.e., for which the set of all finite singular values is bounded. Then  $ f^n(z) \not\to \infty$ as $n \to \infty$ for any $z$ in the Fatou set of $f$. In particular, the function $f$ does not have any BWD.
			\label{B-NoBWD}
		\end{theorem}
		\begin{Remark}
			For each $p>0$ and an entire function  $f\in\mathcal{B}$,  it is known that  $ f^{np}(z) \not \to \infty$ as $n \to \infty$ for any $z $ in the Fatou set of $f$. But this is not true for meromorphic functions of bounded type. 
		\end{Remark}
		There  are sufficient conditions  for meromorphic functions with finitely many poles ensuring the non-existence of BWD. These are based on the fact that the zeros (or pre-images of any non-exceptional point) are separated by annuli with increasing modulus.
		In 2002, Zheng proved a result in this direction. For a complex number $\alpha$ and a meromorphic function $f$, the $\alpha-$value points are the pre-images of $\alpha$ under $f$.
		\begin{theorem}(Corollary 3, \cite{Zheng 2002})\label{Zheng 02}\label{Thm 4.6}
			Let $f$ be a meromorphic function with finitely many poles and $\{z_n\}_{n\geq 1}$ is the sequence of all distinct $\alpha$-value points of $f$ for some  $\alpha$ such that  $|z_{n+1}|>|z_n|$ and $$\sup_{n\geq 1}{\frac{|z_{n+1}|}{|z_n|}<\infty} ,$$ then $f$ has no BWD.
		\end{theorem}
		Recall that $\lambda(f)$ and $\rho(f)$ denote the exponent of convergence of zeros and order of a meromorphic function $f$, respectively.  The next result is for those functions for which the strict inequality $\lambda(f)< \rho(f)$ holds.
		\begin{theorem}(Theorem 1, \cite{CaoWang 2002})\label{Cao-Wang-NoBWD}
			Let  $g$ be an entire function satisfying $\lambda(g)<\rho(g)$ and  $P$ be a polynomial including constants. Then  all the Fatou components of the function  
			$f(z)=g(z)+P(z)$ are simply connected, and hence $f$ does not have  any BWD. In particular,  the function $g$ does not have any BWD.
		\end{theorem}
		Theorem~\ref{Cao-Wang-NoBWD} is applied to construct a class of examples.
		\begin{example}\label{Lem5.1}
			If $P$ and $Q$ are two non-constant polynomials with the same degree and the same leading coefficient, then $e^{P(z)}+e^{Q(z)}$ does not have any BWD.
		\end{example}
		
		\begin{proof}
			Let  $\deg(P)=\deg(Q)=d$ and $f(z)=e^{P(z)}+e^{Q(z)}$. Since the order of $f$ is $d$ and  $\deg(P-Q) < d$, in view of Theorem~\ref{Cao-Wang-NoBWD}, it is enough to show that the exponent of convergence of zeros of $f$ is at most $\deg(P-Q)$.
			
			Let  \begin{equation}\label{Equ 22}
				P(z)-Q(z)=Az^m+o(|z|^m)
			\end{equation} with $A\neq 0$ and $m<d$, where $o(|z|^m)$ is a function of $|z|^m$ that goes to $0$ as $|z| \to \infty$. Note that the zeros of $f$ are the solutions of $e^{ P(z)-Q(z)}=-1$, i.e.,  $ P(z)-Q(z)=(2k+1)\pi i$ for some integer $k$. Let the sequence of all such non-zero solutions be denoted by $\{a_k\}_{k>0}$. Then $A (a_k)^m +o(|a_k|^m)=(2k+1)\pi i$ and therefore $\lim_{ k\to \infty} \frac{(2k+1)\pi i}{a_k ^m}=A$. For $\epsilon < \frac{|A|}{2}$, there is a $k_0$ such that $| \frac{(2k+1)\pi i}{a_k ^m}-A | < \epsilon$ for all $k >k_0$. This gives that $\frac{|A|}{2} < |\frac{(2k+1)\pi i}{a_k ^m}|< \frac{3|A|}{2}$. In other words, $$c_1 |k|^{\frac{1}{m}}< |a_k| <c_2 |k|^{\frac{1}{m}} ~\mbox{ for all}~ k >k_0~\mbox{where}~ c_1=\left(\frac{4 \pi}{3|A|}\right)^{\frac{1}{m} }~\mbox{and}~c_2= \left(\frac{6 \pi}{|A|}\right)^{\frac{1}{m}}.$$ 
			
			Consequently, we have $ \frac{1}{c_{2}^\lambda} \sum_{k=0}^{\infty}\frac{1}{k^\frac{\lambda}{m}} \leq   \sum_{k=0}^{\infty}\frac{1}{|a_k|^\lambda}\leq\frac{1}{c_{1}^\lambda} \sum_{k=0}^{\infty}\frac{1}{k^\frac{\lambda}{m}}$. Thus the series $\sum_{k=0}^{\infty}\frac{1}{|a_k|^\lambda}$ converges if and only  if   $\lambda>m$. Therefore, $\lambda(f) =\deg(P-Q)$. This completes the proof.
		\end{proof}
		
		\begin{Remark}
			There are entire functions of the form $f(z)=g(z)+P(z)$ (where $P$ is a polynomial) without having any BWD, although the condition $\lambda(g)<\rho(g)\leq\infty$ is not satisfied by them. This can be seen by taking $g(z)=e^z$ in Theorem~\ref{Cao-Wang-NoBWD} since the order of $f$ is $1$ whereas the exponent of convergence of its zeros is also $1$ (This can be seen using the argument used in the previous example).
		\end{Remark}
		In 2024,  Cao et al. gave an important property of zeros of an entire function, which becomes a necessary condition for the existence of a BWD.
		
		\begin{theorem}(Theorem 1.4, \cite{CaoWangZhao 2024}\label{Thm36})\label{Thm 4.8}
			If an entire function $f$ has either only finitely many zeros or a sequence of distinct zeros $\{a_n\}_{n\geq 1}$ among all its zeros satisfying
			\begin{equation}\label{eqn7}
				\limsup_{n\to\infty}\frac{\log|a_{n+1}|}{\log|a_n|}=1,
			\end{equation}
			then all the Fatou components of $f$ are simply connected. In particular, there is no BWD of $f$.
		\end{theorem}
		The above theorem can be restated as: if an entire function $f$ has a BWD, then $f$ has infinitely many zeros, say $a_n$, and those satisfy $\limsup_{n\to\infty}\frac{\log|a_{n+1}|}{\log|a_n|}>1$. This is because $ \log|a_{n+1}| \geq  \log|a_n|$ is possible for all but at most finitely many values of $n$.
			
		\begin{Remark}
			In Theorem~\ref{Thm 4.8}, the sequence $a_n$ is not required to satisfy $|a_{n+1}|>|a_n|$ unlike $z_n$ in Theorem \ref{Zheng 02}.  
		\end{Remark}
		There are two results similar to Theorem~\ref{Thm 4.8}, reported in the same paper.
		\begin{theorem}(Theorem 1.1, \cite{CaoWangZhao 2024})
			Let $f$ and $\phi$ be two entire functions such that $M(r,\phi)\leq M(r, f)^\alpha$ for some $\alpha\in(0,1)$ and for all large $r$. If $f(z)-\phi(z)$ has either only finitely many zeros or a sequence of distinct zeros $\{z_n\}_{n\geq 1}$ among all its zeros satisfying Equation (\ref{eqn7}). Then there is no BWD of $f$.
			\label{CW2024}
		\end{theorem}
		
		In Theorem~\ref{CW2024}, the function $\phi$ can be the zero function.
		
		\begin{theorem}(Theorem 1.3, \cite{CaoWangZhao 2024})
			Let $f(z)=s(z)g(z)+\phi(z)$, where $s(z)$ is a periodic entire function, $g(z)$ is a non-zero entire function, and $\phi(z)$ is an entire function such that $M(r,\phi)\leq M(r, sg)^\alpha$ for some constant $\alpha\in(0,1)$ and for all large $r$. If $s(z)$ has zeros, or $s(z)$ has no zeros but $M(r, g)\leq M(r, sg)^\beta$ for a constant $\beta\in(0,1)$ and for all large $r$, then $f$ has no BWD.
		\end{theorem}
		The following theorem demonstrates that the simple connectivity of all Fatou components can be deduced from the distribution of certain points within the periodic or pre-periodic Fatou components. The proof follows a similar approach as that used in proving Theorem \ref{Thm36}.
		\begin{theorem}(Theorem 1.6, \cite{CaoWangZhao 2024})
			If there exist points $\{b_n\}_{n\geq 1}$	in a simply connected Fatou component, or in the Julia set of an entire function $f$ such that $|b_{n+1}|>|b_n|$ and $\limsup_{n\to\infty}\frac{\log|b_{n+1}|}{\log|b_n|}=1$, then $f$ has no BWD.
			\label{CWZ2024}
		\end{theorem}
		A simply connected Fatou component in Theorem~\ref{CWZ2024} can be a pre-periodic, periodic or wandering domain  of $f$.
		\par We finish this section with an interesting result by Wang and Yang.
		\begin{theorem}(Theorem 3(1). \cite{WangYang 2003})
			Let $f$ and $g$ be two commuting entire functions, i.e.,  $f\circ g=g\circ f$ and $g=f+b$ where $b\neq 0$, then none of $f$ or $g$   has any BWD.
			\label{permutable-BWD}
		\end{theorem}
		There are several examples satisfying the hypotheses of Theorem~\ref{permutable-BWD}. To see it, let $h$ be a periodic entire function with period $w_0$. Then consider $f(z)= z+\lambda h(z)$ for $\lambda \neq 0$ and $g(z)= f(z)+w_0$.
		It is seen that $f$ and $g$ commute. It is interesting to note that the resulting functions may not be of bounded type. This is the case for $f(z)=\sin z$ or $e^z$. 
		\par
		Necessary and/or sufficient conditions for the existence of BWDs are obtained for various classes of meromorphic functions. 
		Corollary~\ref{BWD-iff-MCFC} provides a necessary and sufficient condition for entire functions. A similar result is also known for meromorphic functions with finitely many poles (see Theorem~\ref{FP-BWD-IFF}). Sufficient conditions ensuring BWDs for functions with finitely many poles are also known (see Theorems~\ref{Thm2.7} and \ref{MC-IFF-BWD-FP}). On this background, the following question is natural.
		\begin{Question}
			Find a sufficient condition for meromorphic functions with infinitely many poles to have BWDs. 
		\end{Question}
	
		\par
		\textbf{Acknowledgement}
		The authors are thankful to Janina Kotus for drawing our attention to \cite{bk2}, where meromorphic function with BWD appeared for the first time.
		\section{Disclosure statement}
		The authors report that there is no competing interest to declare.
		\section{Funding}
		The first author is supported by the University Grants Commission, Govt. of India through a Senior research fellowship.
		\section{Data Availability statement} Data sharing not applicable to this article as no datasets were generated or analyzed during the current study.
		

\begin{thebibliography}{99}
			\bibitem{Baker 1963}
			I. N. Baker, \emph{Multiply-connected domains of normality in iteration theory}, Math. Z. 81 (1963) 206–214.
			\bibitem{Baker 1975}
			I. N. Baker,  \emph{The domains of normality of an entire function}, Ann. Acad. Sci. Fenn. Ser. A I Math.1 (1975), no. 2, 277–283.
			\bibitem{Baker 1976}
			I. N. Baker,  \emph{An entire function which has wandering domains}, J. Austral. Math. Soc. Ser. A 22 (1976), no. 2, 173–176.
			\bibitem{Baker 1984}
			I. N. Baker,\emph{ Wandering domains in the iteration of entire functions}, Proc. London Math. Soc. (3) 49 (1984), no. 3, 563–576.
			\bibitem{Baker 1985} I. N. Baker,\emph{ Some entire functions with multiply-connected wandering domains}, Ergodic Theory Dynam. Systems, 5 (1985), no. 2, 163–169.
			
			\bibitem{Baker 1988} I. N. Baker,\emph{ Infinite limits in the iteration of entire functions}, Ergodic Theory Dynam. Systems 8 (1988), no. 4, 503–507.
			
			\bibitem{Baker Dominguez 2000} I. N. Baker, P. Dom\'inguez, \emph{ Some connectedness properties of Julia sets}, Complex Variables Theory Appl. 41 (2000), no.~4, 371--389.
			\bibitem{bk2}
			I.~N. Baker, J.~Kotus, L.~Yinian, \emph{Iterates of meromorphic
				functions {II}: Examples of wandering domains}, Journal of the London Mathematical Society 2(2),  (1991), 267-278.
			\bibitem{bk3}
			I.~N. Baker, J.~Kotus, L.~Yinian, \emph{Iterates of meromorphic
				functions {III}: Preperiodic domains}, Ergodic Theory Dynam. Systems 11 (1991) 603--618.
			\bibitem{Beardon-Book} A. F. Beardon, \emph{Iteration of Rational Functions}, Grad. Texts in Math. 132,  Springer-Verlag, New York, 1991.
			
			\bibitem{Benini rippon Stallard 2015} A.~M. Benini, P.~J. Rippon, G.~M. Stallard,  \emph{Permutable entire functions and multiply connected wandering domains}, Adv. Math. 287 (2016), 451--462.
			\bibitem{Berg 1993} W. Bergweiler, \emph{Iteration of meromorphic functions}, Bull. Amer. Math. Soc. (N.S.) 29 (1993), no. 2, 151–188.
			\bibitem{Berg 2010-simply-multiply-wd}W. Bergweiler,  \emph{An entire function with simply and multiply connected wandering domains}, Pure Appl. Math. Q. {  7} (2011), no.~1, 107--120.
			
			%
			
			
			\bibitem{BergEreme 1995}
			W. Bergweiler,  A. Eremenko, \emph{On the singularities of the inverse to a meromorphic function of finite order}, Rev. Mat. Iberoamericana 11 (1995), no. 2, 355–373.
			\bibitem{BergHinkk 1999}
			W. Bergweiler,  A. Hinkkanen, \emph{On semiconjugation of entire functions}, Math. Proc. Cambridge Philos. Soc. 126 (1999), no. 3, 565–574.
			\bibitem{Bergetal 1993}
			W. Bergweiler, M. Haruta, H.  Kriete, H. Meier, N. Terglane, \emph{On the limit functions of iterates in wandering domains}, Ann. Acad. Sci. Fenn. Ser. A I Math. 18 (1993), no. 2, 369–375.
			\bibitem{BergRippStall 2008}
			W. Bergweiler, P.~J. Rippon and G.~M. Stallard, \emph{Dynamics of meromorphic functions with direct or logarithmic singularities}, Proc. Lond. Math. Soc. (3) 97 (2008), no.~2, 368--400.
			\bibitem{BergRippStall 2013}
			W. Bergweiler, P. J.  Rippon,   G. M. Stallard, \emph{Multiply connected wandering domains of entire functions}, Proc. Lond. Math. Soc. (3) 107 (2013), no. 6, 1261–1301.
			\bibitem{BergZheng 2011}
			W. Bergweiler, J. H. Zheng, \emph{On the uniform perfectness of the boundary of multiply connected
				wandering domains},  J. Aust. Math. Soc. 91 (2011) 289–311.
			
			\bibitem{Bishop-2018}
			C. J. Bishop,  \emph{A transcendental Julia set of dimension 1}, Invent. math. 212 (2018), 407–460. 
			
			\bibitem{Boas-1954} R.~P. Boas Jr., {\it Entire functions}, Academic Press, New York, 1954.
			\bibitem{Bolsch 1999}
			A. Bolsch, \emph{Periodic Fatou components of meromorphic functions}, Bull. London Math. Soc. 31 (1999), no. 5, 543–555.
			\bibitem{CaoWang 2002}
			C. L. Cao, Y. F. Wang, \emph{On the simple connectivity of Fatou components}, Acta Math. Sin. (Engl. Ser.) 18 (2002), no. 4, 625–630.
			\bibitem{CaoWangZhao 2024}
			C. Cao, Y. Wang, H. Zhao, \emph{Topological properties of certain iterated entire maps}, Anal. Math. Phys. 14 (2024), no. 2, Paper No. 18, 12 pp.
			\bibitem{ChakraChakrabortyNayak 2016}
			T. K. Chakra,  G. Chakraborty, T. Nayak,  \emph{Baker omitted value}, Complex Var. Elliptic Equ. 61 (2016), no. 10, 1353–1361.
			\bibitem{DSU 2026}
			K. Darch, L. Pardo-Sim\'on, B. U\v{c}akar, \emph{Bounded-orbit wandering domains do not exist}, arXiv.2609.30103, (2026).
			\bibitem{DasNayak 2024}
			S. Das, T. Nayak, \emph{Sum of the exponential and a polynomial:
				singular values and Baker wandering
				domains}, Complex Var. Elliptic Equ. 70 (2025), no. 10, 1831–1847.
			\bibitem{Dominguez 1997}
			P. Dom\'inguez-Soto, \emph{Connectedness properties of Julia sets of transcendental entire functions},  Complex Variables Theory Appl. {  32} (1997), no.~3, 199--215.
			\bibitem{Dominguez 1998}
			P. Dom\'inguez, \emph{Dynamics of transcendental meromorphic functions}, Ann. Acad. Sci. Fenn. Math. 23 (1998), no. 1, 225–250.
			\bibitem{Eremenko 1989}
			A. E. Eremenko, \emph{On the iteration of entire functions}, Dynamical systems and ergodic theory (Warsaw, 1986), 339–345.
			Banach Center Publ., 23
			PWN—Polish Scientific Publishers, Warsaw, 1989.
			\bibitem{EremenkoLyubich 1984}
			A. Eremenko and M.~Y. Lyubich, \emph{Iterations of entire functions}, Dokl. Akad. Nauk SSSR 279 (1984), no.~1, 25--27.
			\bibitem{EremenkoLyubich 1992}
			A. Eremenko, M. Y. Lyubich, \emph{Dynamical properties of some classes of entire functions}, Ann. Inst. Fourier (Grenoble)
			42 (1992), no. 4, 989–1020.  
			\bibitem{EremenkoLyubich 1987}
			A. Eremenko and M.~Y. Lyubich, \emph{Examples of entire functions with pathological dynamics}, J. London Math. Soc. (2) 36 (1987), no.~3, 458--468.
			
			\bibitem{gustavo-2022}
			G. R. Ferreira, \emph{Multiply connected wandering domains of meromorphic functions: internal dynamics and connectivity},  J. Lond. Math. Soc. 106(3)  (2022), 1897–1919.
			
			\bibitem{gustavo-2024}
			G. R. Ferreira, \emph{Multiply connected wandering domains of meromorphic functions: the pursuit of uniform internal dynamics}, Ergod. Theory Dyn. Syst. 44(3) (2024), 727–748. 
			
			
			\bibitem{GhoraNayakSahoo 2021}
			S. Ghora,  T. Nayak, S.  Sahoo,  \emph{On Fatou sets containing Baker omitted value}, J. Dynam. Differential Equations 35 (2021), no. 3, 2621–2639.
			\bibitem{Rempe2017}
			L. Rempe-Gillen, D. Sixsmith,  \emph{Hyperbolic entire functions and the Eremenko-Lyubich class:
				Class B or not class B?}, Math. Z. 286 (2017), no. 3–4, 783–800.
			\bibitem{Hadamard1896}
			J. Hadamard, \emph{Sur les fonctions entieres}, Bull. Soc. Math. France, 24, (1896) pp. 186–187.
			\bibitem{Herring-1998}
			M.~E. Herring, \emph{Mapping properties of Fatou components}, Ann. Acad. Sci. Fenn. Math. { 23} (1998), no.~2, 263--274.
			\bibitem{Iversen 1914}
			F. Iversen, \emph{Recherches sur les fonctions inverses des fonctions meromorphes}, These de Helsingfors, 1914.
			\bibitem{Kisaka2025}M. Kisaka, \emph{Julia components of transcendental entire functions with multiply-connected wandering domains}, Comput. Methods Funct. Theory {  25} (2025), no.~2, 317--327.
			\bibitem{KisakaShishikura 2008}
			M. Kisaka, M. Shishikura, \emph{On multiply connected wandering domains of entire functions}, London Math. Soc. Lecture Note Ser., 348, Cambridge University Press, Cambridge, 2008.
			%
			
			\bibitem{NayakZheng 2011}
			T. Nayak, J. H. Zheng, \emph{Omitted values and dynamics of transcendental meromorphic functions}, J. Lond. Math. Soc. vol. 83, no. 1, (2011), pp. 121–136.
			\bibitem{Nayak 2015}
			T. Nayak,  \emph{On Fatou components and omitted values}, Contemp. Math., 639, American Mathematical Society, Providence, RI (2015), 349–358.
			\bibitem{Noop 1947}
			K. Knopp,  \emph{ Theory and Application of Infinite Series}, Hafner Publishing Company, New York, 1947.
			
			\bibitem{Simon Sixsmith 2024} 
			L. Pardo-Sim\'on,  D.~J. Sixsmith, \emph{Wandering domains with nearly bounded orbits}, Proc. Amer. Math. Soc.  152 (2024), no.~10, 4311--4323.
			\bibitem{PRW 2026}
			N. Prochorov, L. Rempe, J. Waterman, \emph{Absence of bounded-orbit wandering domains}, arXiv.2609.28279, (2026).
			
			
			
			
			\bibitem{RipponStallard 1999}
			P. J. Rippon, G. M. Stallard, \emph{Iteration of a class of hyperbolic meromorphic functions}, Proc. Amer. Math. Soc. 127 (1999), 3251--3258.
			
			\bibitem{RipponStallard 2000}
			P. J. Rippon, G. M. Stallard, \emph{On sets where iterates of a meromorphic function zip towards infinity}, Bull. London Math. Soc. 32 (2000), no. 5, 528–536.
			\bibitem{RipponStallard 2005}
			P. J. Rippon, G. M. Stallard, \emph{On questions of Fatou and Eremenko}, Proc. Amer. Math. Soc. 133 (2005), no. 4, 1119–1126.
			\bibitem{RipponStallard 2005(1)}
			P. J. Rippon, G. M. Stallard, \emph{Escaping points of meromorphic functions with a finite number of poles}, J. Anal. Math. 96 (2005), 225–245.
			\bibitem{RipponStallard 2008}
			P. J. Rippon, G. M. Stallard, \emph{On multiply connected wandering domains of meromorphic functions}, J. Lond. Math. Soc. (2) 77 (2008), no. 2, 405–423.
			
			\bibitem{RRRS 2011}
			G. Rottenfusser, J. R\"{u}ckert, L. Rempe, D. Schleicher, \emph{Dynamic rays of bounded-type entire
				functions}, Ann. of Math. 173 (2011) 77–125.
			\bibitem{Sixsmith2018} D. J. Sixsmith,  \emph{Dynamics in the Eremenko-Lyubich class}, Conform. Geom. Dyn. 22 (2018), 185–224.
			\bibitem{Sullivan 1985}
			D. Sullivan, \emph{Quasiconformal homeomorphisms and dynamics. I. Solution of the Fatou-Julia problem on wandering domains}, Ann. of Math. (2), vol. 122, no. 3 (1985) pp. 401–418.
			\bibitem{WangYang 2003}
			X. Wang, C. C. Yang, \emph{On the Fatou components of two permutable transcendental entire functions}, J. Math. Anal. Appl. 278 (2003), no. 2, 512–526.
			\bibitem{asymptoticvalue-book}
			G. H. Zhang,  \emph{Theory of entire and meromorphic functions. Deficient and asymptotic values and singular directions}, Translated from the Chinese by Yang, C.-C., Translations of Mathematical Monographs, 122, Amer. Math. Soc., Providence, RI, 1993.
			\bibitem{Zheng 2000}
			J. H. Zheng, \emph{Singularities and wandering domains in iteration of meromorphic functions}, Illinois
			J. Math. 44 (2000) 520–530.
			\bibitem{Zheng 2002}
			J. H. Zheng, \emph{On uniformly perfect boundary of stable domains in iteration of meromorphic functions II},
			Math. Proc. Cambridge Philos. Soc. 132 (2002), no. 3, 531–544.
			\bibitem{Zheng 2006}
			J. H. Zheng, \emph{On multiply-connected Fatou components in iteration of meromorphic functions},
			J. Math. Anal. Appl. 313 (2006), no. 1, 24–37.
		\end{thebibliography}
	\end{document}